\documentclass[10.5pt,reqno]{article}
\usepackage{mathrsfs}
\usepackage[colorlinks=true]{hyperref}
\usepackage{amsmath,amssymb,amsthm}
\usepackage[english]{babel}
\newtheorem{theorem}{Theorem}[section]
\newtheorem{lemma}[theorem]{Lemma}
\newtheorem{corollary}[theorem]{Corollary}

\newtheorem{proposition}[theorem]{Proposition}
\newtheorem{remark}[theorem]{Remark}
\numberwithin{equation}{section}

\title{On the Cauchy problem for the inhomogeneous nonlinear Schr\"{o}dinger equation with inverse-power potential}
\author{{\bf JinMyong An, JinMyong Kim$^*$ and OkByol Kim}\\
\footnotesize{Faculty of Mathematics, {\bf Kim Il Sung} University, Pyongyang, Democratic People's Republic of Korea}\\
\footnotesize{$^*$ Corresponding Author: jm.kim0211@ryongnamsan.edu.kp.}
}
\date{}
\begin{document}
\maketitle
\begin{abstract}
In this paper, we study the Cauchy problem for the inhomogeneous nonlinear Schr\"{o}dinger equation with inverse-power potential
\[iu_{t} +\Delta u-c|x|^{-a}u=\pm |x|^{-b} |u|^{\sigma } u,\;
\;(t,x)\in \mathbb R\times\mathbb R^{d},\]
where $d\in \mathbb N$, $c\in \mathbb R$, $a,b>0$ and $\sigma>0$.
First, we establish the local well-posedness in the fractional Sobolev spaces $H^s(\mathbb R^d)$ with $s\ge 0$ by using contraction mapping principle based on the Strichartz estimates in Sobolev-Lorentz spaces. Next, the global existence and blow-up of $H^1$-solution are investigated. Our results extend the known results in several directions.
\end{abstract}


\section{Introduction}

In this paper, we study the Cauchy problem for inhomogeneous nonlinear
Schr\"{o}dinger equation with inverse-power potential:
\begin{equation} \label{GrindEQ__1_1_}
\left\{\begin{array}{l} {iu_{t}+\Delta u-c|x|^{-a}u=\lambda|x|^{-b}
|u|^{\sigma } u,\;(t,x)\in\mathbb R\times\mathbb R^{d},}
\\ {u\left(0, x\right)=u_{0}(x),} \end{array}\right.
\end{equation}
where $u:\mathbb R\times \mathbb R^{d} \to \mathbb C$, $u_{0}:\mathbb R^{d} \to \mathbb C$, $d\in \mathbb N$, $c\in \mathbb R$, $0<a<2$,  $0<b<2$, $\sigma>0$ and $\lambda=\pm1$. The parameters $\lambda=1$ and $\lambda=-1$ corresponds to the defocusing and focusing cases respectively.

The equation \eqref{GrindEQ__1_1_} appears in a variety of physical settings, for example, in nonlinear optical systems with spatially dependent interactions (see e.g. \cite{BPVT07,BPST03,KSWW75,KMVBT17} and the references therein).
The case $c=b=0$ is the classic nonlinear Schr\"{o}dinger (NLS) equation which has been extensively studied over the last three decades (see e.g. \cite{C03,CFH11,DYC13, LP15, WHHG11} and the references therein).

When $b=0$ and $c\neq 0$, we have the NLS equation with inverse-power potential, which has also been widely studied in the past decades.
See e.g. \cite{D18J,KMVZZ18,OSY12,Y21} for the inverse-square potential $a=2$ and \cite{AHI24,D21A,FO19, GWY22,MZZ22,M20} for the slowly decaying potentials $0<a<2$.

Moreover, the inhomogeneous NLS equation without potential (i.e. \eqref{GrindEQ__1_1_} with $b>0$ and $c=0$) has also attracted a lot of interest in recent years.
We refer the reader to \cite{AT21, AK211, AK212,AK23,AK232,AKC22,GS08,G17} for the local well-posedness and small data global well-posedness in $H^s$ with $s\ge 0$ and \cite{AK232,BL23,CC21, CF22,D18N, DK21, F16,G12, GM21} for the global well-posedness and blow-up in the energy space $H^{1}$.

Recently, the inhomogeneous NLS equation with inverse-square potential, i.e. \eqref{GrindEQ__1_1_} with $a=2$ has also been investigated by \cite{AJK23,AMK24,BS23,BGS23,CCF23,CG21,JAK21,S16}.

As in the study of the NLS equation with potential, in this paper, we mainly focus on the local and global well-posedness as well as blow-up for inhomogeneous NLS equation with slowly decaying potential, i.e. \eqref{GrindEQ__1_1_} with $0<a,b<2$.

\subsection{Known results}

In this subsection, we recall the known results for the inhomogeneous NLS equation (i.e. \eqref{GrindEQ__1_1_} with $c=0$ and $b>0$) and the NLS equation with inverse-power potential (i.e. \eqref{GrindEQ__1_1_} with $b=0$ and $0<a<2$).

\subsubsection{Known results for \eqref{GrindEQ__1_1_} with $c=0$ and $b>0$}
Let us recall the known results for the inhomogeneous NLS equation without potential, namely,
\begin{equation} \label{GrindEQ__1_2_}
\left\{\begin{array}{l} {iu_{t}+\Delta u=\lambda |x|^{-b}
|u|^{\sigma } u,\;(t,x)\in\mathbb R\times\mathbb R^{d},}
\\ {u\left(0,x\right)=u_{0}(x),} \end{array}\right.
\end{equation}
where $d\in \mathbb N$, $0<b<\min\{2,d\}$, $\sigma>0$ and $\lambda=\pm1$.
The inhomogeneous NLS equation \eqref{GrindEQ__1_2_} is invariant under the scaling,
\begin{equation} \label{GrindEQ__1_3_}
u_{\mu}(t,x):=\mu^{\frac{2-b}{\mu}}u(\mu^{2}t,\mu x),~\mu>0.
\end{equation}
An direct computation shows that
$$
\left\|u_{\mu}(0)\right\|_{\dot{H}^{s}}=\mu^{s-\frac{d}{2}+\frac{2-b}{\mu}}\left\|u_0\right\|_{\dot{H}^{s}}.
$$
We thus define the critical exponents
\begin{equation}\label{GrindEQ__1_4_}
s_{\rm c}:=\frac{d}{2}-\frac{2-b}{\sigma}
\end{equation}
and
\begin{equation}\label{GrindEQ__1_5_}
\gamma_{\rm c}:=\frac{1-s_{\rm c}}{s_{\rm c}}=\frac{4-2b-(d-2)\sigma}{d\sigma-4+2b}.
\end{equation}
Putting
\begin{equation} \label{GrindEQ__1_6_}
\sigma_{\rm c}(s,b):=
\left\{\begin{array}{cl}
{\frac{4-2b}{d-2s},} ~&{{\rm if}~s<\frac{d}{2},}\\
{\infty,}~&{{\rm if}~s\ge \frac{d}{2},}
\end{array}\right.
\end{equation}
we can easily see that $s>s_{\rm c}$ is equivalent to $\sigma<\sigma_{\rm c}(s,b)$. If $s<\frac{d}{2}$, then $s=s_{\rm c}$ is equivalent to $\sigma=\sigma_{\rm c}(s,b)$.
For initial data $u_{0}\in H^{s}(\mathbb R^{d})$, we say that the Cauchy problem \eqref{GrindEQ__1_2_} is $H^{s}$-critical if $0\le s<\frac{d}{2}$ and $\sigma=\sigma_{\rm c}(s,b)$.
If $s\ge 0$ and $\sigma<\sigma_{\rm c}(s,b)$, then the problem \eqref{GrindEQ__1_2_} is said to be  $H^{s}$-subcritical.
Especially, if $\sigma =\sigma_{\rm c}(0,b)$, then the problem is also known as mass-critical.
If $\sigma =\sigma_{\rm c}(1,b)$ with $d\ge 3$, it is also called energy-critical.
If $\sigma_{\rm c}(0,b)<\sigma<\sigma_{\rm c}(1,b)$, it is called mass-supercritical and energy-subcritical (or intercritical).
The inhomogeneous NLS equation \eqref{GrindEQ__1_2_} has formally the  conservation of mass and energy, which are defined respectively by
\begin{equation} \label{GrindEQ__1_7_}
M(u(t)):=\int_{\mathbb R^{d}}{| u(t,x)|^{2}dx}=M(u_{0}),
\end{equation}
\begin{equation} \label{GrindEQ__1_8_}
E_{b}(u(t)):=\frac{1}{2}\int_{\mathbb R^{d}}{|\nabla u(t,x)|^2 dx}+\frac{\lambda
}{\sigma+2}\int_{\mathbb R^{d}}{|x|^{-b}\left|u(t,x)\right|^{\sigma+2}dx}=E_b(u_{0}).
\end{equation}
\smallskip

\textbf{1) $H^s$-subcritical case}

\smallskip

Using the energy method developed by \cite{C03}, Genoud-Stuart \cite{GS08} proved that \eqref{GrindEQ__1_2_} is locally well-posed in $H^1(\mathbb R^d)$ if $d\in \mathbb N$, $0<b<\min\{2,d\}$ and $0<\sigma<\sigma_{\rm c}(1,b)$.
However, this energy method cannot be applied to established the local well-posedness for \eqref{GrindEQ__1_2_} in the general Sobolev spaces $H^s$.
Later, Guzm\'{a}n \cite{G17} used the contraction mapping principle based on the Strichartz estimates in Sobolev spaces to prove that \eqref{GrindEQ__1_2_} is locally well-posedness in $H^s(\mathbb R^d)$ with $0\le s\le \min\{1,\frac{d}{2}\}$ if $d\in \mathbb N$, $0<b<\min\{2,\frac{d}{3}\}$ and $0<\sigma <\sigma _{\rm c}(s,b)$.
It was also proved that the above local $H^s$-solution is extended globally in time if $\sigma>\frac{4-2b}{d}$ and the initial data is sufficiently small.
The well-posedness results of \cite{G17} were later extended by \cite{AK211,AK212,AKC22}.
More precisely, the authors in \cite{AK211} proved that \eqref{GrindEQ__1_2_} is locally well-posed in $H^{s}(\mathbb R^d)$, if $0\le s<\min \{d,\frac{d}{2} +1\}$, $0<b<\min \{2,d-s,1+\frac{d-2s}{2}\}$, $0<\sigma <\sigma _{\rm c}(s,b)$ and the following regularity assumption for the nonlinear term is further satisfied\footnote[1]{For $s\in\mathbb R$, $\left\lceil s\right\rceil $ denotes the minimal integer which is larger than or equal to $s$.}:
\begin{equation} \label{GrindEQ__1_9_}
\sigma~\textrm{is an even integer, or}~ \sigma\ge \lceil s\rceil-1.
\end{equation}
It was also proved in \cite{AK212} that the above local $H^s$-solution is extended globally in time if $\sigma>\frac{4-2b}{d}$ and the initial data is sufficiently small. Furthermore, the authors in \cite{AKC22} obtained some standard continuous dependence result for \eqref{GrindEQ__1_2_} in $H^s$-subcritical case.

On the other hand, the global existence as well as blow-up of $H^1$-solution for the energy-subcritical problem \eqref{GrindEQ__1_2_} with $0<\sigma <\sigma _{\rm c}(1,b)$ have also been widely studied.

In the defocusing case $\lambda=1$, it is easily proved that any local $H^1$-solution to energy-subcritical problem \eqref{GrindEQ__1_2_} with $0<\sigma<\sigma_{\rm c}(1,b)$ is extended globally in time by using the conservation laws and the blow-up alternative.

The focusing problem \eqref{GrindEQ__1_2_} with $\lambda=-1$ were studied by \cite{CC21, CF22, D18N, DK21, F16, G12, GM21}.
In the mass-subcritical case $0<\sigma<\sigma_{\rm c}(0,b)$, it was proved in \cite{GS08} that focusing problem \eqref{GrindEQ__1_2_} with $d\in \mathbb N$ and $0<b<\min\{2,d\}$ is globally well-posed in $H^1$. The mass-critical case $\sigma=\sigma_{\rm c}(0,b)$ was studied in \cite{G12,D18N,CF22}.
More precisely, Genoud \cite{G12} proved that the focusing, mass-critical problem \eqref{GrindEQ__1_2_} with $d\in \mathbb N$ and $0<b<\min\{2,d\}$ is globally well-posed in $H^1$ if $\left\|u_0\right\|_{L^2}<\left\|Q\right\|_{L^2}$,
where $Q$ is the unique positive radial solution of the ground state equation
\begin{equation}\nonumber
\Delta Q-Q+|x|^{-b}|Q|^{\frac{4-2b}{d}}Q=0.
\end{equation}
It was also proved in \cite{D18N,CF22} that any $H^1$-solution for the focusing, mass-critical problem \eqref{GrindEQ__1_2_} with initial data $u_0\in H^1$ satisfying $E_b(u_0)<0$ blows up in finite time. The mass-supercritical and energy-subcritical case $\sigma_{\rm c}(0,b)<\sigma<\sigma_{\rm c}(1,b)$ was investigated in \cite{CC21, D18N, DK21, F16,GM21}.
Farah \cite{F16} showed the global existence for \eqref{GrindEQ__1_2_} with $d\in \mathbb N$, $0<b<\min\{2,d\}$ by assuming $u_0\in H^1$ and
\begin{equation}\label{GrindEQ__1_10_}
E_b(u_{0})[M(u_{0})]^{\gamma_{\rm c}}<E_b(Q)[M(Q)]^{\gamma_{\rm c}},
\end{equation}
\begin{equation}\label{GrindEQ__1_11_}
  \left\|\nabla u_{0}\right\|_{L^{2}}\left\|u_{0}\right\|_{L^{2}}^{\gamma_{\rm c}}
  <\left\|\nabla Q\right\|_{L^{2}}\left\|Q\right\|_{L^{2}}^{\gamma_{\rm c}},
  \end{equation}
where $Q$ is the unique positive radial solution to the elliptic equation
\begin{equation}\label{GrindEQ__1_12_}
\Delta Q-Q+|x|^{-b}|Q|^{\sigma}Q=0.
\end{equation}
He also proved the finite time blow-up for \eqref{GrindEQ__1_2_} with $u_0\in \Sigma:=H^1\cap L^{2}(|x|^2 dx)$ satisfying \eqref{GrindEQ__1_10_} and
\begin{equation}\label{GrindEQ__1_13_}
  \left\|\nabla u_{0}\right\|_{L^{2}}\left\|u_{0}\right\|_{L^{2}}^{\gamma_{\rm c}}
  >\left\|\nabla Q\right\|_{L^{2}}\left\|Q\right\|_{L^{2}}^{\gamma_{\rm c}}.
\end{equation}
The latter result was extended to the radial data in \cite{D18N}, and to the non-radial data in \cite{BL23,DK21}. Note that the uniqueness of positive radial solution to \eqref{GrindEQ__1_12_} was established in \cite{Y91,G11,T84}.
The long time dynamics for the focusing problem \eqref{GrindEQ__1_2_} with $\sigma_{\rm c}(0,b)<\sigma<\sigma_{\rm c}(1,b)$ and
\begin{equation}\nonumber
E_b(u_{0})[M(u_{0})]^{\gamma_{\rm c}}\ge E_b(Q)[M(Q)]^{\gamma_{\rm c}},
\end{equation}
were also investigated by \cite{CC21,DK21}.

\smallskip

\textbf{2) $H^s$-critical case}

\smallskip

Aloui-Tayachi \cite{AT21} developed a local well-posedness theory for \eqref{GrindEQ__1_2_} in both of $H^{s}$-subcritical case and $H^{s}$-critical case.
More precisely, they used the Sobolev-Lorentz spaces to prove that \eqref{GrindEQ__1_2_} is locally well-posed in $H^{s}$ if $d\in \mathbb N$, $0\le s\le 1$, $s<\frac{d}{2}$, $0<b<\min\{2,d-2s\}$ and $0<\sigma\le\frac{4-2b}{d-2s}$.
However, they didn't treat the case $1<s<\frac{d}{2}$ with $d\ge 3$. This case was later investigated by \cite{AK23}.
More precisely, they used the fractional Hardy inequality to prove that the $H^{s}$-critical problem \eqref{GrindEQ__1_2_} with $d\ge 3$, $1<s<\frac{d}{2}$, $0<b<1+\frac{d-2s}{2}$ and $\sigma =\sigma_{\rm c}(s,b)$ is locally well-posed in $H^{s}(\mathbb R^{d})$ if \eqref{GrindEQ__1_9_} and one of the conditions in the following system are further satisfied:
\begin{equation} \label{GrindEQ__1_14_}
\left\{\begin{array}{l}
{s\in \mathbb N,~d\ge 3~\textrm{and}~b<4s/d,}\\
{s\notin \mathbb N,~d\ge 4~\textrm{and}~b<6s/d-1,}\\
{s\notin \mathbb N,~d=3~\textrm{and}~b<1.}
\end{array}\right.
\end{equation}
However, one can easily see that the conditions on $b$ in the $H^{s}$-critical case obtained by \cite{AT21,AK23} are not so good as in the $H^{s}$-subcritical case obtained by \cite{AK211,AK212,AKC22}.
The local well-posedness results of \cite{AT21,AK23} for \eqref{GrindEQ__1_2_} in the $H^s$-critical case $\sigma=\sigma_{\rm c}(s,b)$ were finally extended in \cite{AK232} to the full range of $0\le s< \frac{d}{2}$ and $0<b<\min \{2,d-s,1+\frac{d-2s}{2}\}$. The authors in \cite{AK232} also showed the blow-up of the $H^1$-solution for the focusing, energy-critical problem \eqref{GrindEQ__1_2_} with initial data $u_0\in H^1$ satisfying

$\cdot$ $E_b(u_0)<0$ or

$\cdot$ $E_b(u_0)\ge 0$, $E_b(u_0)\le E_b(W_b)$ and $\left\|\nabla u_{0}\right\|_{L^{2}}>\left\|\nabla W_{b}\right\|_{L^{2}}$,\\
where
\begin{equation}\label{GrindEQ__1_15_}
W_{b}(x):=\frac{\left[\varepsilon(d-b)(d-2)\right]^{\frac{d-2}{4-2b}}}
{\left(\varepsilon+|x|^{2-b}\right)^{\frac{d-2}{2-b}}}.
\end{equation}
with $\varepsilon>0$. We also refer the reader to \cite{GM21} and the references therein for the global well-posedness and scattering for the focusing, energy-critical problem \eqref{GrindEQ__1_2_}.

\subsubsection{Known results for \eqref{GrindEQ__1_1_} with $c\neq 0$ and $b=0$}

Let us recall the known results for the NLS equation with inverse-power potential:
\begin{equation} \label{GrindEQ__1_16_}
\left\{\begin{array}{l} {iu_{t}+\Delta u-c|x|^{-a}u=\lambda |u|^{\sigma } u,\;(t,x)\in\mathbb R\times\mathbb R^{d},}
\\ {u\left(0, x\right)=u_{0}(x),} \end{array}\right.
\end{equation}
where $d\in \mathbb N$, $c\neq 0$, $0<a<2$, $\sigma>0$ and $\lambda=\pm1$.

The global-in-time Strichartz estimates for $e^{-itH_c}$ with $H_c=-\Delta +c|x|^{-\sigma}$, $c>0$ and $0<a<2$ were established in \cite{M20} in dimension $d\ge 3$. Miao-Zhang-Zheng \cite{MZZ22} studied the global well-posedness, finite time blow-up and scattering in the energy space $H^1(\mathbb R^3)$ for the NLS equation with coulomb potential, i.e. \eqref{GrindEQ__1_16_} with $a=1$.
Later, Dinh \cite{D21A} extend the results of \cite{MZZ22} to the NLS equation with a general class of inverse-power potentials (i.e. \eqref{GrindEQ__1_16_} with $0<a<2$) and higher dimensions.
More precisely, using the energy method developed by \cite{C03}, he showed that \eqref{GrindEQ__1_16_} is locally well-posed in $H^1(\mathbb R^d)$ if $d\in \mathbb N$, $0<a<\min\{2,d\}$ and $0<\sigma<\tilde{\sigma}_{\rm c}(1)$, where
\begin{equation} \label{GrindEQ__1_17_}
\tilde{\sigma}_{\rm c}(s):=\sigma_{\rm c}(s,0)=
\left\{\begin{array}{cl}
{\frac{4}{d-2s},} ~&{{\rm if}~s<\frac{d}{2},}\\
{\infty,}~&{{\rm if}~s\ge \frac{d}{2}.}
\end{array}\right.
\end{equation}
Using Strichartz estimates in Lorentz spaces (see Lemma \ref{lem 2.13.}), he also gave the alternative proof of the local well-posedness for \eqref{GrindEQ__1_16_} with $d\ge 3$, $0<\sigma<\frac{4}{d-2}$ and
\begin{equation} \label{GrindEQ__1_18_}
\left\{\begin{array}{cl}
{0<a<\frac{3}{2},} ~&{{\rm if}~d=3,}\\
{0<a<2,}~&{{\rm if}~d\ge4.}
\end{array}\right.
\end{equation}
Note that although the latter result is weaker than the former result on the validity of $d$ and $a$, it gives more information on the solution, for instance, one knows that the solutions to \eqref{GrindEQ__1_16_} satisfy $u\in L_{\textrm{loc}}^{p}((-T_{\textrm{min}},T_{\textrm{max}}),H_q^{1})$ for any admissible pair $(p,q)$.
He also proved that \eqref{GrindEQ__1_16_} with $d\ge 3$, $0<a<2$ and $\sigma=\frac{4}{d-2}$ is locally well-posed in $H^1$ by using the the global in time Strichartz estimates obtain by \cite{M20}.
Furthermore, he obtained some global well-posedness, finite time blow-up and scattering results in the energy space $H^1$ for \eqref{GrindEQ__1_2_} with $0<a<2$.
We also refer the reader to \cite{BPVT07, FO19, GWY22} for further results.
\subsection{Main results}
\subsubsection{Local well-posedness in $H^s$}

Using the energy method developed by \cite{C03}, we have the following local well-posedness in $H^1$ for \eqref{GrindEQ__1_1_} with $\sigma<\sigma_{\rm c}(1,b)$.

\begin{proposition}\label{prp 1.1.}
Let $d\in \mathbb N$, $c\in \mathbb R$, $\lambda=\pm 1$, $0<a,b<\min \{2,d\}$ and $0<\sigma<\sigma_{c}(1,b)$.
Then for any $u_0\in H^1$, there exist $T_{*},T^{*}\in (0,\infty]$ and a unique maximal solution
\begin{equation} \nonumber
u\in C((-T_{*} , T^{*}), H^{1})\cap  C^{1}((-T_{*} , T^{*}), H^{-1}),
\end{equation}
of \eqref{GrindEQ__1_1_}. If $T^{*}<\infty$ (resp. if $T_{*}<\infty$), then $\left\|u_{0}\right\|_{H^1}\rightarrow \infty$ as $t\uparrow T^{*}$ (resp. as $t\downarrow -T_{*}$).
Moreover, the following mass and energy of $u(t)$ are conserved:
\begin{equation} \label{GrindEQ__1_19_}
M(u(t)):=\int_{\mathbb R^{d}}{| u(t,x)|^{2}dx},
\end{equation}
\begin{equation} \label{GrindEQ__1_20_}
E_{b,c}(u(t)):=\int_{\mathbb R^{d}}{\left(\frac{1}{2}|\nabla u(t,x)|^2+\frac{c}{2}|x|^{-a}|u(t,x)|^{2}+\frac{\lambda
}{\sigma+2}|x|^{-b}\left|u(t,x)\right|^{\sigma+2}\right)dx}.
\end{equation}

\end{proposition}
\begin{proof}
Noticing that $0<a<\min\{2,d\}$, we can see that the potential $c|x|^{-a}$ belongs to $L^r(\mathbb R^d)+L^{\infty}(\mathbb R^d)$ for some $r>\max\{1,\frac{d}{2}\}$.
Hence, using Theorem 4.3.1, Example 3.2.11 of \cite{C03} and Appendix K of \cite{GS08}, we can get the desired result.
\end{proof}

Using the contraction mapping principle based on Strichartz estimates in Lorentz spaces, we also have the following local well-posedness in $H^s$ for \eqref{GrindEQ__1_1_}.
\begin{theorem}\label{thm 1.2.}
Let $d\in \mathbb N$, $c\in \mathbb R$, $0\le s<\frac{d}{2}$ , $0\le b<\min \{2, d-s,1+\frac{d}{2}-s \}$, $0<a<\min \{2, d-s,1+\frac{d}{2}-s \}$ and $0<\sigma\le \frac{4-2b}{d-2s}$. If $\sigma$ is not an even integer, assume further $\sigma>\left\lceil s\right\rceil-1$.
Then for any $u_{0}\in H^{s}(\mathbb R^{d}) $, there exist $T_{\max }=T_{\max }(u_0)>0$ and $T_{\min }=T_{\min }(u_0)>0$ such that \eqref{GrindEQ__1_1_} has a unique, maximal solution $u$ satisfying
\begin{equation} \label{GrindEQ__1_21_}
u\in C((-T_{\min },T_{\max } ),H^{s} )\cap L^{p}_{\rm loc} ((-T_{\min } ,T_{\max } ),H_{q,2}^{s} ),
\end{equation}
for any admissible pair $(p,q)$. Moreover, the above solution $u$ depends continuously on the initial data $u_{0}$ in the following sense. There exists $0<T<T_{\max } ,\, T_{\min } $ such that if $u_{0,n} \to u_{0} $ in $H^{s} (\mathbb R^{d})$ and if $u_{n} $ denotes the solution of \eqref{GrindEQ__1_1_} with the initial data $u_{0,n} $, then $0<T<T_{\max }(u_{0,n}),T_{\min } (u_{0,n})$ for all sufficiently large $n$ and $u_{n} \to u$ in $L^{p} ([-T,T],L^{q,2}(\mathbb R^{d}))$ as $n\to \infty $ for any admissible pair $(p,q)$. Furthermore, if $s>0$, then $u_{n} \to u$ in $L^{p} ([-T,T],H_{q,2}^s(\mathbb R^{d}))$ as $n\to \infty $ for all admissible pair $(p,q)$ and all $\varepsilon >0$. In particular, $u_{n} \to u$ in $C([-T,T],H^{s-\varepsilon})$ as $n\to \infty $ for all $\varepsilon >0$.
\end{theorem}
\begin{remark}\label{rem 1.3.}
\textnormal{Theorem \ref{thm 1.2.} extends the known results for \eqref{GrindEQ__1_1_} in the following directions.
\begin{itemize}
  \item When $b>0$, it extends the results of \cite{AT21,AK232} for \eqref{GrindEQ__1_1_} with $c=0$ to \eqref{GrindEQ__1_1_} with $c\in \mathbb R$.
  \item When $b>0$ and $c=0$, it improves the results of \cite{AT21} by extending the validity of $s$ and $b$. It also extends the result of  \cite{AK232} by giving the unified proof for both of $H^s$-subcritical case $\sigma<\frac{4-2b}{d-2s}$ and $H^s$-critical case $\sigma=\frac{4-2b}{d-2s}$.
  \item When $b=0$ and $c\neq 0$, it improves the result of \cite{D21A} (see Proposition 3.3 of \cite{D21A}) by extending the validity of $s$.
\end{itemize}}
\end{remark}

We also have the following standard continuous dependence result.
\begin{theorem}\label{thm 1.4.}
Let $d\in \mathbb N$, $c\in \mathbb R$, $0\le s<\frac{d}{2}$ , $0<a,b<\min \{2, d-s,1+\frac{d}{2}-s \}$ and $0<\sigma\le \frac{4-2b}{d-2s}$. If $\sigma$ is not an even integer, assume that
\begin{equation} \nonumber
\left\{\begin{array}{ll}
{\sigma>\left\lceil s\right\rceil-1,} ~&{{\rm for}~\sigma<\frac{4-2b}{d-2s},}\\
{\sigma\ge \left\lceil s\right\rceil,}~&{{\rm for}~\sigma=\frac{4-2b}{d-2s}.}
\end{array}\right.
\end{equation}
If $s<1$, in addition, suppose further that $\sigma>1$.
Then for any given $u_{0} \in H^{s} (\mathbb R^{d})$, the corresponding solution $u$ of \eqref{GrindEQ__1_1_} in Theorem \ref{thm 1.2.} depends continuously on the initial data $u_{0}$ in the following sense.
For any interval $[-S,T]\subset(-T_{\min }(u_{0}),T_{\max }(u_{0}))$, and every admissible pair $(p,q)$, if $u_{0,n} \to u_{0} $ in $H^{s} (\mathbb R^{d})$ and if $u_{n} $ denotes the solution of \eqref{GrindEQ__1_1_} with the initial data $u_{0,n} $, then $u_{n} \to u$ in $L^{p} ([-S,T],H_{q,2}^{s} (\mathbb R^{d} ))$ as $n\to \infty $.
In particular, $u_{n} \to u$ in $C([-S,T],H^{s})$.
\end{theorem}

\begin{remark}\label{rem 1.5.}
\textnormal{Theorem \ref{thm 1.4.} extends the standard continuous dependence results of \cite{AK232,AKC22} for $c=0$ to $c\in \mathbb R$.
  Furthermore, when $c=0$, it gives the unified proof for both of $H^s$-subcritical case $\sigma<\frac{4-2b}{d-2s}$ and $H^s$-critical case $\sigma=\frac{4-2b}{d-2s}$.}
\end{remark}

\subsubsection{Global existence and blow-up of $H^{1}$-solution}
Concerning the global existence of $H^{1}$-solution, we have the following result.
\begin{theorem}[Global existence]\label{thm 1.6.}
Let $d\in \mathbb N$, $c\in \mathbb R$, $\lambda=\pm1$, $0<a,b<\min \{2,d\}$ and $0<\sigma<\sigma_{c}(1)$.
If one of the following conditions is satisfied, then the corresponding solution $u$ of \eqref{GrindEQ__1_1_} with initial data $u_0\in H^1(\mathbb R^d)$ in Proposition \ref{prp 1.1.} is a global one.
\begin{enumerate}
  \item $\lambda=1$.
  \item $\lambda=-1$ and $\sigma<\frac{4-2b}{d}$.
  \item $\lambda=-1$, $\sigma=\frac{4-2b}{d}$ and $\left\|u_0\right\|_{L^2}<\left\|Q\right\|_{L^2}$.
  \item $\lambda=-1$, $\frac{4-2b}{d}<\sigma<\sigma_{\rm c}(1,b)$, $c\ge0$ and
       \begin{equation}\label{GrindEQ__1_22_}
       E_{b,c}(u_{0})[M(u_{0})]^{\gamma_{\rm c}}<E_{b}(Q)[M(Q)]^{\gamma_{\rm c}},~
       \left\|\nabla u_{0}\right\|_{L^{2}}\left\|u_{0}\right\|_{L^{2}}^{\gamma_{\rm c}}
        <\left\|\nabla Q\right\|_{L^{2}}\left\|Q\right\|_{L^{2}}^{\gamma_{\rm c}},
       \end{equation}
\end{enumerate}
where $Q$ is the unique positive radial solution to the elliptic equation \eqref{GrindEQ__1_12_} and $\gamma_{\rm c}$ is as in \eqref{GrindEQ__1_5_}.
\end{theorem}

In the focusing case $\lambda=-1$, we also have the following blow-up result.
\begin{theorem}[Blow-up]\label{thm 1.7.}
Let $d\in \mathbb N$, $c\ge0$, $\lambda=-1$, $0<a,b<\min \{2,d\}$ and $\frac{4-2b}{d}\le \sigma\le \sigma_{c}(1,b)$ with $\sigma<\infty$.
Let $u$ be the solution to \eqref{GrindEQ__1_1_} defined on the maximal forward time interval of existence $[0,T^{*})$.
\begin{enumerate}
  \item (Mass-critical case: $\sigma=\frac{4-2b}{d}$.) If $E_{b,c}(u_0)<0$, the corresponding solution $u$ blows up in finite time, i.e. $T^{*}<\infty$.
  \item (Intercritical case: $\frac{4-2b}{d}<\sigma<\sigma_{\rm c}(1,b)$.) If $E_{b,c}(u_0)<0$ or if not, we assume that
       \begin{equation}\label{GrindEQ__1_23_}
       E_{b,c}(u_{0})[M(u_{0})]^{\gamma_{\rm c}}<E_{b}(Q)[M(Q)]^{\gamma_{\rm c}},~
       \left\|\nabla u_{0}\right\|_{L^{2}}\left\|u_{0}\right\|_{L^{2}}^{\gamma_{\rm c}}
        >\left\|\nabla Q\right\|_{L^{2}}\left\|Q\right\|_{L^{2}}^{\gamma_{\rm c}}.
       \end{equation}
       Then the corresponding solution $u$ blows up in finite or infinite time, i.e. either $T^{*}<\infty$ or $T^{*}=\infty$ and there exists a time sequence $t_{n}\to \infty$ such that $\left\|\nabla u(t_{n})\right\|_{L^2}\to \infty$ as $n\to \infty$.
       Moreover, if we assume in addition that $\sigma<\frac{4}{d}$, then the corresponding solution $u$ blows up in finite time, i.e. $T^{*}<\infty$.
  \item (Energy-critical case: $\sigma=\frac{4-2b}{d-2}$ with $d\ge 3$.) If $E_{b,c}(u_0)<0$ or if not, we assume that
  \begin{equation}\label{GrindEQ__1_24_}
  E_{b,c}(u_{0})<E_{b}(W_{b}),~\left\|\nabla u_{0}\right\|_{L^2}>\left\|\nabla W_{b}\right\|_{L^2},
  \end{equation}
where $W_{b}$ is as in \eqref{GrindEQ__1_15_}. Then the corresponding solution $u$ blows up in finite or infinite time. Moreover, if we assume in addition that $b>\frac{4}{d}$, then the corresponding solution $u$ blows up in finite time.
\end{enumerate}
\end{theorem}
\begin{remark}\label{rem 1.8.}
\textnormal{Theorems \ref{thm 1.6.} and \ref{thm 1.7.} extend global existence and blow up results of \cite{AK232,BL23,CF22,DK21,F16} for \eqref{GrindEQ__1_1_} with $c=0$ to $c\in \mathbb R$ or $c\ge 0$. Furthermore, when $c=0$ and $\sigma=\frac{4-2b}{d-2}$, Item 3 of Theorem \ref{thm 1.7.} improves the blow-up result of \cite{AK232}. In fact, the authors in \cite{AK232} only obtained the finite or infinite blow-up result for non-radial data (see Theorem 1.10 of \cite{AK232}). However, in Item 3 of Theorem \ref{thm 1.7.}, we also obtained the finite time blow-up result for non-radial data and $b>\frac{4}{d}$.}
\end{remark}
In addition, we have the following blow-up result for mass-critical inhomogeneous NLS equation with inverse-square potential, i.e. \eqref{GrindEQ__1_1_} with $a=2$.
\begin{theorem}\label{thm 1.9.}
Let $d\ge 3$, $a=2$, $0<b<2$, $\lambda=-1$, $\sigma=\frac{4-2b}{d}$ and $c>-\left(\frac{d-2}{2}\right)^2$. Let $u$ be the solution to \eqref{GrindEQ__1_1_} defined on the maximal forward time interval of existence $[0,T^{*})$. If $E_{b,c}(u_0)<0$, then the corresponding solution $u$ blows up in finite time. A similar statement holds for negative times.
\end{theorem}
\begin{remark}\label{rem 1.10.}
\textnormal{Theorem \ref{thm 1.9.} improves the blow-up results of \cite{AMK24, CG21} for mass-critical inhomogeneous NLS equation with inverse-square potential. In fact, when $E_{b,c}(u_0)<0$, Campos-Guzm\'{a}n \cite{CG21} proved the finite time blow-up only for radial or finite variance data (see Theorem 1.2 of \cite{CG21}). Later, the authors in \cite{AMK24} showed the finite or infinite time blow-up for non-radial data (see Proposition 1.5 of \cite{AMK24}).}
\end{remark}
This paper is organized as follows. In Section \ref{sec 2.}, we recall some notation and give some preliminary results related to our problem. In Section \ref{sec 3.}, we prove Theorem \ref{thm 1.2.}. Theorem \ref{thm 1.4.} is proved in Section \ref{sec 4.}. In Section \ref{sec 5.}, we prove Theorem \ref{thm 1.6.}. in Section \ref{sec 6.}, we prove Theorems \ref{thm 1.7.} and \ref{thm 1.9.}.

\section{Preliminaries}\label{sec 2.}
\subsection{Notation} Let us introduce some notation used in this paper. Throughout the paper, $\mathscr{F}$ denotes the Fourier transform, and the inverse Fourier transform is denoted by $\mathscr{F}^{-1}$. $C>0$ stands for a positive universal constant, which may be different at different places. $a\lesssim b$ means $a\le Cb$ for some constant $C>0$. $a\sim b$ expresses $a\lesssim b$ and $b\lesssim a$.
Given normed spaces $X$ and $Y$, $X\hookrightarrow Y$ means that $X$ is continuously embedded in $Y$.
For $p\in \left[1,\;\infty \right]$, $p'$ denotes the dual number of $p$, i.e. $1/p+1/p'=1$.
For $s\in \mathbb R$, we denote by $\left\lceil s\right\rceil$ the minimal integer which is larger than or equal to $s$.
As in \cite{WHHG11}, for $s\in \mathbb R$ and $1<p<\infty $, we denote by $H_{p}^{s} (\mathbb R^{d} )$ and $\dot{H}_{p}^{s} (\mathbb R^{d} )$ the nonhomogeneous Sobolev space and homogeneous Sobolev space, respectively. The norms of these spaces are given as
$$
\left\| f\right\| _{H_{p}^{s} (\mathbb R^{d} )} :=\left\| (I-\Delta)^{s/2} f\right\| _{L^{p} (\mathbb R^{d} )} , ~\left\| f\right\| _{\dot{H}_{p}^{s} (\mathbb R^{d} )} :=\left\| (-\Delta)^{s/2} f\right\| _{L^{p} (\mathbb R^{d} )},
$$
where $(I-\Delta)^{s/2}f =\mathscr{F}^{-1} \left(1+|\xi|^{2} \right)^{s/2} \mathscr{F}f$ and $(-\Delta)^{s/2}f =\mathscr{F}^{-1} |\xi|^{s} \mathscr{F}f$. As usual, we abbreviate $H_{2}^{s} (\mathbb R^{d} )$ and $\dot{H}_{2}^{s} (\mathbb R^{d} )$ as $H^{s} (\mathbb R^{d} )$ and $\dot{H}^{s} (\mathbb R^{d} )$, respectively. For
$0<p,\; q\le \infty $, we denote by $L^{p,q} (\mathbb R^{d})$ the Lorentz space (see e.g. \cite{G14}).
The quasi-norms of these spaces are given by
$$
\left\|f\right\|_{L^{p,q} (\mathbb R^{d})}:=\left(\int_{0}^{\infty}{\left(t^{\frac{1}{p}}f^{*}(t)\right)^{k}
\frac{dt}{t}}\right)^{\frac{1}{\tilde{q}}},~~\textnormal{when}~~0<q<\infty,
$$
$$
\left\|f\right\|_{L^{p,\infty} (\mathbb R^{d})}:=\sup_{t>0}t^{\frac{1}{p}}f^{*}(t),~~\textnormal{when}~~q=\infty,
$$
where $f^{*}(t)=\inf\left\{\tau:M^{d}\left(\left\{x:|f(x)|>\tau\right\}\right)\le t\right\}$, with $M^{d}$ being the Lebesgue measure in $\mathbb R^{d}$. Note that $L^{p,q} \left(\mathbb R^{d}
\right)$ is a quasi-Banach space for $0<p,\; q\le \infty $.  When $1<p<\infty$ and $1\le q \le \infty$,  $L^{p,q} \left(\mathbb R^{d}
\right)$ can be turned into a Banach space via an equivalent norm.  In particular $L^{p,p} (\mathbb R^{d})=L^{p}
(\mathbb R^{d})$, while $L^{p,\infty}$ corresponds to weak $L^p$ space.
These spaces are natural in the context of \eqref{GrindEQ__1_1_} since $|x|^{-b}\in L^{\frac{d}{b},\infty}(\mathbb R^d)$.
In general, we have the embedding $L^{p,q}\hookrightarrow L^{p,r}$ for $q<r$.
Note also that $\left\| \left|f\right|^{r} \right\| _{L^{p,q} } =\left\| f\right\| _{L^{pr,qr}}^{r}$ for $1\le p,\;r<\infty $, $1\le q\le \infty $.
We also have the H\"{o}lder's inequality in Lorentz spaces (see e.g. \cite{O63}):
$$\left\| fg\right\| _{L^{p,q} } \lesssim \left\| f\right\|
_{L^{p_{1},q_{1} } } \left\| g\right\| _{L^{p_{2},q_{2} } } .$$
for $1<p,p_{1},p_{2}<\infty $ and $1\le q,q_{1} ,q_{2}\le \infty $ satisfying

$$
\frac{1}{p}=\frac{1}{p_{1}} +\frac{1}{p_{2}} ,~\frac{1}{q}=\frac{1}{q_{1} } +\frac{1}{q_{2} }.
$$
For $I\subset \mathbb R$ and $\gamma \in \left[1,\;\infty \right]$, we will use the space-time mixed space $L^{\gamma } \left(I,X\left(\mathbb R^{d}
\right)\right)$ whose norm is defined by
\[\left\| f\right\|_{L^{\gamma } \left(I,\;X(\mathbb R^{d})\right)}
=\left(\int _{I}\left\| f\right\| _{X(\mathbb R^{d})}^{\gamma } dt
\right)^{\frac{1}{\gamma } } ,\]
with a usual modification when $\gamma =\infty $, where $X(\mathbb R^{d})$
is a normed space on $\mathbb R^{d} $. If there is no confusion, $\mathbb R^{d} $ will be omitted in various function spaces.

\subsection{Sobolev-Lorentz spaces}

Throughout the paper, we mainly use the Sobolev-Lorentz spaces.
In this subsection, we recall some useful facts about the Sobolev-Lorentz spaces.
See \cite{AT21,AK232,AKR24, HYZ12} for example.

For $s\in \mathbb R$, $1<p<\infty$ and $1\le q \le \infty$, the (nonhomogeneous) Sobolev-Lorentz space ${H}^{s}_{p,q}(\mathbb R^{d})$ is defined as the set of tempered distribution $f\in \mathscr{S}'(\mathbb R^d)$ such that $(I-\Delta)^{s/2}f \in L^{p,q}(\mathbb R^{d})$, equipped with the norm
$$\left\|f\right\|_{{H}^{s}_{p,q}(\mathbb R^{d})}:=\left\|(I-\Delta)^{s/2}f \right\|_{L^{p,q}(\mathbb R^{d})}.$$
The homogeneous Sobolev-Lorentz space $\dot{H}^{s}_{p,q}(\mathbb R^{d})$ is defined as the set of equivalence classes of distribution $f\in \mathscr{S}'(\mathbb R^d)/{\mathscr{P}(\mathbb R^d)}$ such that $(-\Delta)^{s/2}f \in L^{p,q}(\mathbb R^{d})$, equipped with the norm
$$\left\|f\right\|_{\dot{H}^{s}_{p,q}(\mathbb R^{d})}:=\left\|(-\Delta)^{s/2}f\right\|_{L^{p,q}(\mathbb R^{d})},$$
where $\mathscr{P}(\mathbb R^d)$ denotes the set of polynomials with $d$ variables.
\begin{lemma}[\cite{AK232}]\label{lem 2.1.}
Let $s\ge 0$, $1<p<\infty$ and $1\le q_{1}\le q_{2}\le \infty$. Then we have

$(a)$ $\dot{H}^{s}_{p,1}\hookrightarrow \dot{H}^{s}_{p,q_{1}} \hookrightarrow\dot{H}^{s}_{p,q_{2}} \hookrightarrow \dot{H}^{s}_{p,\infty}$.

$(b)$ $\dot{H}^{s}_{p,p}=\dot{H}^{s}_{p}$.
\end{lemma}
\begin{lemma}[\cite{AK232}]\label{lem 2.2.}
Let $s\ge 0$, $1<p<\infty $, $1\le q\le \infty$ and $v=s-[s]$. Then we have
\[ \left\|f\right\|_{\dot{H}_{p,q}^{s} }\sim \sum _{\left|\alpha \right|=[s]}\left\|D^{\alpha }f\right\| _{\dot{H}_{p,q}^{v}}.\]
\end{lemma}
\begin{lemma}[\cite{AKR24}]\label{lem 2.3.}
Let $s\ge 0$, $1<p<\infty $ and $1\le q\le \infty$. Then we have $H^{s}_{p,q}=L^{p,q}\cap \dot{H}^{s}_{p,q}$ with
\[\left\|f\right\|_{H^{s}_{p,q}}\sim \left\|f\right\|_{L^{p,q}}+\left\|f\right\|_{\dot{H}^{s}_{p,q}}.\]
\end{lemma}

\begin{lemma}[\cite{AKR24}]\label{lem 2.4.}
Let $-\infty < s_{2} \le s_{1} <\infty $ and $1<p_{1} \le p_{2} <\infty $ with $s_{1} -\frac{d}{p_{1} } =s_{2} -\frac{d}{p_{2} } $. Then for any $1\le q\le \infty$, there holds the embeddings:
$$
\dot{H}_{p_{1},q}^{s_{1} } \hookrightarrow \dot{H}_{p_{2},q}^{s_{2}},~H_{p_{1},q}^{s_{1} } \hookrightarrow H_{p_{2},q}^{s_{2}}
$$
\end{lemma}
\begin{lemma}[\cite{AKR24}]\label{lem 2.5.}
Let $s\in \mathbb R$, $\varepsilon\ge 0$, $1<p<\infty $ and $1\le q\le \infty$. Then we have $H^{s+\varepsilon}_{p,q}\hookrightarrow H^{s}_{p,q}$.
\end{lemma}
As an immediate consequence of Lemmas \ref{lem 2.4.} and \ref{lem 2.5.}, we have the following.
\begin{corollary}[\cite{AKR24}]\label{cor 2.6.}
Let $-\infty < s_{2} \le s_{1} <\infty $ and $1<p_{1} \le p_{2} <\infty $ with $s_{1} -\frac{d}{p_{1} } \ge s_{2} -\frac{d}{p_{2} } $. Then for any $1\le q\le \infty$, there holds the embedding: $H_{p_{1},q}^{s_{1} } \hookrightarrow H_{p_{2},q}^{s_{2} }$.
\end{corollary}
Using Lemma \ref{lem 2.4.} and the fact that $|x|^{-\gamma}\in L^{\frac{d}{\gamma},\infty}$, we also have the following Hardy inequality under Lorentz norms.
\begin{corollary}[Hardy inequality under Lorentz norms, \cite{HYZ12}]\label{cor 2.7.}
Let $1<p<\infty$ and $1\le q\le \infty $ with $0<\gamma <\frac{d}{p}$. Then we have
$$
\left\||x|^{-\gamma}f\right\|_{L^{p,q}}\lesssim \left\|f\right\|_{\dot{H}_{p,q}^{\gamma}}.
$$
\end{corollary}
We also recall the fractional product rule and chain rule in Sobolev-Lorentz spaces.
\begin{lemma}[Fractional product rule in Sobolev-Lorentz spaces, \cite{CN16}]\label{lem 2.7.}
Let $s\ge 0$, $1<p,p_{1},p_{2},p_{3},p_{4} <\infty$ and $1\le q,q_{1},q_{2},q_{3},q_{4} \le \infty$. Assume that
\begin{equation}\nonumber
\frac{1}{p} =\frac{1}{p_{1} }+\frac{1}{p_{2}}=\frac{1}{p_{3} }+\frac{1}{p_{4}},~\frac{1}{q} =\frac{1}{q_{1} }+\frac{1}{q_{2}}=\frac{1}{q_{3} }+\frac{1}{q_{4}}.
\end{equation}
Then we have
\begin{equation}\nonumber
\left\|fg\right\| _{\dot{H}_{p,q}^{s} } \lesssim \left\|f\right\| _{\dot{H}_{p_1,q_1}^{s} } \left\| g\right\| _{L^{p_{2},q_{2}}} +\left\| f\right\| _{L^{p_{3},q_{3}}}
\left\| g\right\| _{\dot{H}_{p_4,q_4}^{s} }.
\end{equation}
\end{lemma}

\begin{lemma}[Fractional chain rule in Sobolev-Lorentz spaces, \cite{AT21}]\label{lem 2.9.}
Suppose $F\in C^{1} (\mathbb C\to \mathbb C)$ and $0< s\le1$. Then for $1<p,\;p_{1},\;p_{2} <\infty $ and $1\le q,\;q_{1},\; q_{2}< \infty $ satisfying
$$
\frac{1}{p} =\frac{1}{p_{1} } +\frac{1}{p_{2} },~\frac{1}{q} =\frac{1}{q_{1} } +\frac{1}{q_{2} },
$$
we have
\begin{equation} \nonumber
\left\|F(u)\right\| _{\dot{H}_{p,q}^{s}} \lesssim \left\| F'(u)\right\| _{L^{p_{1},q_{1}}} \left\| u\right\| _{\dot{H}_{p_2,q_2}^{s}}.
\end{equation}
\end{lemma}

Lemma \ref{lem 2.9.} which holds for $0<s\le 1$ was then extended to $s>0$ by \cite{AKR24}.
\begin{lemma}[\cite{AKR24}]\label{lem 2.10.}
Let $s>0$ and $F\in C^{\left\lceil s\right\rceil } \left(\mathbb C\to \mathbb C\right)$. Then for $1<p,p_{1_k},p_{2_k},p_{3_k}<\infty$ and $1\le q,q_{1_k},q_{2_k},q_{3_k}<\infty$ satisfying
\begin{equation} \nonumber
\frac{1}{p}=\frac{1}{p_{1_k}}+\frac{1}{p_{2_k}}+\frac{k-1}{p_{3_k}},
~\frac{1}{q}=\frac{1}{q_{1_k}}+\frac{1}{q_{2_k}}+\frac{k-1}{q_{3_k}},~k=1,2,\ldots, \lceil s\rceil,
\end{equation}
we have
\begin{equation} \nonumber
\left\|F(u)\right\| _{\dot{H}_{p,q}^{s}} \lesssim \sum_{k=1}^{\lceil s\rceil}\left\| F^{(k)}(u)\right\| _{L^{p_{1_k},q_{1_k}}} \left\|u\right\|_{\dot{H}_{p_{2_k},q_{2_k}}^s}\left\|u\right\|_{L^{p_{3_k},q_{3_k}}}^{k-1}.
\end{equation}
\end{lemma}

As an immediate consequence of Lemma \ref{lem 2.10.}, we have the following useful nonlinear estimates in Sobolev-Lorentz spaces $\dot{H}_{p,q}^{s}$ with $s\ge 0$.
\begin{corollary}[\cite{AKR24}]\label{cor 2.11.}
Let $s\ge 0$ and $\sigma >0$. If $\sigma$ is not an even integer, assume further $\sigma >\lceil s\rceil -1$. Then for $1<p,p_{1},p_{2}<\infty$ and $1\le q,q_{1},q_{2}<\infty$ satisfying
\begin{equation} \nonumber
\frac{1}{p}=\frac{\sigma}{p_{1}}+\frac{1}{p_{2}},~\frac{1}{q}=\frac{\sigma}{q_{1}}+\frac{1}{q_{2}},
\end{equation}
we have
\begin{equation} \nonumber
\left\||u|^{\sigma}u\right\| _{\dot{H}_{p,q}^{s}} \lesssim \left\|u\right\| _{L^{p_{1},q_{1}}}^{\sigma} \left\|u\right\|_{\dot{H}_{p_{2},q_{2}}^{s}}.
\end{equation}
\end{corollary}

Finally, we recall the interpolation inequality in Sobolev-Lorentz spaces.

\begin{lemma}[Convexity H\"{o}lder's inequality in Sobolev-Lorentz spaces, \cite{AKR24}]\label{lem 2.12.}
Let $1<p,\;p_{i} <\infty $, $1\le q,\;q_{i} <\infty $, $0\le \theta _{i} \le 1$, $s\ge 0$, $s_{i}\ge 0\; (i=1,\;\ldots ,\;N)$, $\sum _{i=1}^{N}\theta _{i} =1$, $s=\sum _{i=1}^{N}\theta _{i}  s_{i} $, $1/p=\sum _{i=1}^{N}{\theta _{i}/p_{i}}$ and $1/q=\sum _{i=1}^{N}{\theta _{i}/q_{i}}$.
Then we have $\bigcap _{i=1}^{N}\dot{H}_{p_{i},q_{i} }^{s_{i} }  \subset \dot{H}_{p,q}^{s} $ and for any $f\in \bigcap _{i=1}^{N}\dot{H}_{p_{i},q_{i} }^{s_{i} }  $,
\begin{equation} \nonumber
\left\|f\right\| _{\dot{H}_{p,q}^{s} } \le \prod _{i=1}^{N}\left\| f\right\| _{\dot{H}_{p_{i},q_{i} }^{s_{i} } }^{\theta _{i} }  .
\end{equation}
\end{lemma}

\subsection{Strichartz estimates}
In this subsection, we recall the Strichartz estimates for the Schr\"{o}dinger semi-group $e^{it\Delta}$ in Lorentz spaces, which are the fundamental tool to establish the local well-posedness of the Cauchy problem \eqref{GrindEQ__1_1_} in $H^s$. See, for example, \cite{AT21,D21A,KT98}.
Throughout the paper, a pair $(p,q)$ is said to be admissible (for short $(p,q)\in S$) if
\begin{equation} \label{GrindEQ__2_1_}
\left\{\begin{array}{ll}
{2\le q\le\frac{2d}{d-2}},~&{{\rm if}~d\ge3,}\\
{2\le q<\infty,}~&{{\rm if}~d\le 2,}
\end{array}\right.
\end{equation}
and
\begin{equation} \label{GrindEQ__2_2_}
\frac{2}{p} =\frac{d}{2} -\frac{d}{q} .
\end{equation}
Moreover, $(p,q)\in S_0$ means that $(p,q)\in S$ with
\begin{equation} \nonumber
\left\{\begin{array}{ll}
{2< q<\frac{2d}{d-2}},~&{{\rm if}~d\ge3,}\\
{2< q<\infty,}~&{{\rm if}~d\le 2.}
\end{array}\right.
\end{equation}

\begin{lemma}[Strichartz estimates]\label{lem 2.13.}
Let $d\in \mathbb N$ and $S(t)=e^{it\Delta}$. Then for any admissible pairs $(p,q)$ and $(a,b)$, we have
\begin{equation} \label{GrindEQ__2_3_}
\left\| S(t)\phi \right\| _{L^{p}(\mathbb R,L^{q,2}(\mathbb R^d))} \lesssim\left\| \phi \right\| _{L^{2}(\mathbb R^d) } ,
\end{equation}
\begin{equation} \label{GrindEQ__2_4_}
\left\|\int_{0}^{t}S(t-\tau)f(\tau)d\tau\right\|_{L^{p}(\mathbb R,L^{q,2}(\mathbb R^d))}\lesssim\left\|f\right\|_{L^{a'}(\mathbb R,L^{b',2}(\mathbb R^d))}.
\end{equation}
\end{lemma}
\subsection{Variational Analysis}
In this subsection, we recall the sharp Gagliardo--Nirenberg inequality and the sharp Hardy-Sobolev embedding which are useful to study the global existence and blow-up of $H^1$-solution to \eqref{GrindEQ__1_1_}.
\begin{lemma}[Sharp Gagliardo-Nirenberg inequality, \cite{F16}]\label{lem 2.14.}
Let $d\ge 1$, $0<\sigma<\sigma_{\rm c}(1,b)$, $0<b<2$. Then for $f\in H^{1}$, we have
\begin{equation}\label{GrindEQ__2_5_}
\int{|x|^{-b}|f(x)|^{\sigma+2}dx} \le C_{\rm GN}\left\|\nabla f\right\|_{L^2}^{\frac{d\sigma+2b}{2}}
\left\|f\right\|_{L^{2}}^{\frac{4-2b-\sigma(d-2)}{2}},
\end{equation}
The equality in \eqref{GrindEQ__2_5_} is attained by a function $Q\in H^{1}$, which is the unique positive radial solution to the elliptic equation \eqref{GrindEQ__1_12_}.
\end{lemma}
\begin{remark}\label{rem 2.15.}
\textnormal{We also have the following Pohozaev identities:
\begin{equation}\label{GrindEQ__2_6_}
\left\|Q\right\|_{L^2}^2=\frac{4-2b-(d-2)\sigma}{d\sigma+2b}
\left\|\nabla Q\right\|_{L^2}^{2}=\frac{4-2b-(d-2)\sigma}{2(\sigma+2)}\int{|x|^{-b}|f(x)|^{\sigma+2}dx} .
\end{equation}
In the mass-critical case $\sigma=\frac{4-2b}{d}$, we have
\begin{equation}\label{GrindEQ__2_7_}
C_{\rm GN}=\frac{2-b+d}{d}\left\|Q\right\|_{L^2}^{-\frac{4-2b}{d}}.
\end{equation}
In the intercritical case $\frac{4-2b}{d}<\sigma<\sigma_{\rm c}(1,b)$, we also have
\begin{equation}\label{GrindEQ__2_8_}
C_{\rm GN}=\frac{2(\sigma+2)}{d\sigma+2b}\left(\left\|\nabla Q\right\|_{L^{2}}\left\|Q\right\|_{L^{2}}^{\gamma_{\rm c}}\right)^{-\frac{d\sigma-4+2b}{2}},
\end{equation}
where $\gamma_{\rm c}$ is as in \eqref{GrindEQ__1_5_}.}
\end{remark}
\begin{lemma}[Sharp Hardy-Sobolev embedding, \cite{KP04}]\label{lem 2.16.}
Let $d\ge 3$ and $0< b< 2$. Then we have
\begin{equation}\label{GrindEQ__2_9_}
\left(\int{|x|^{-b}\left|f\right|^{\sigma_{\rm c}(1,b)+2}dx}\right)^{\frac{1}{\sigma_{\rm c}(1,b)+2}}
\le C_{\rm HS} \left\|\nabla f\right\|_{L^2},
\end{equation}
for all $f\in \dot{H}^{1}$, where the sharp Hardy-Sobolev constant $C_{\rm HS}$ defined by
\begin{equation}\label{GrindEQ__2_10_}
C_{\rm HS}=\inf_{f\in \dot{H}^{1}\setminus\left\{0\right\}}
{\frac{\left\|\nabla f\right\|_{L^2}}
{\left(\int{|x|^{-b}\left|f\right|^{\sigma_{\rm c}(1,b)+2}dx}\right)
^{\frac{1}{\sigma_{\rm c}(1,b)+2}}}}.
\end{equation}
is attained by function
\begin{equation}\label{GrindEQ__2_11_}
W_{b}(x):=\frac{\left[\varepsilon(d-b)(d-2)\right]^{\frac{d-2}{4-2b}}}
{\left(\varepsilon+|x|^{2-b}\right)^{\frac{d-2}{2-b}}},
\end{equation}
for all $\varepsilon>0$.
\end{lemma}
Lemma 2.2 in \cite{KP04} also shows that $W_{b}$ solves the equation
\begin{equation}\nonumber
\Delta W_{b}+|x|^{-b}\left|W_{b}\right|^{\sigma_{\rm c}(1,b)}W_{b}=0,
\end{equation}
and satisfies
\begin{equation}\label{GrindEQ__2_12_}
\left\|\nabla W_{b}\right\|_{L^2}^2=\int{|x|^{-b} W_{b}^{\sigma_{\rm c}(1,b)+2}dx}.
\end{equation}
Hence, we have
\begin{equation}\label{GrindEQ__2_13_}
\left\|\nabla W_{b}\right\|_{L^2}^2=\int{|x|^{-b} W_{b}^{\sigma_{\rm c}(1,b)+2}dx}=[C_{\rm HS}]^{-\frac{2(d-b)}{2-b}}
\end{equation}
and
\begin{equation} \label{GrindEQ__2_14_}
E_{b}(W_{b})=\frac{1}{2} \left\|
\nabla W_{b}\right\|_{L^2}^2-\frac{1}{\sigma_{\rm c}(1,b)+2} \int{|x|^{-b} \left|W_{b}\right|^{\sigma_{\rm c}(1,b)+2}dx}=\frac{2-b}{2(d-b)}[C_{\rm HS}]^{-\frac{2(d-b)}{2-b}}.
\end{equation}

\section{Local well-posedness in $H^s$}\label{sec 3.}
In this section, we prove Theorem \ref{thm 1.2.}.
First, we recall the useful fact concerning to the term $|x|^{-b}$ with $b>0$.
\begin{remark}[\cite{G17}]\label{rem 3.1.}
\textnormal{Let $b>0$, $s\ge 0$ and $b+s<d$. Then we have $(-\Delta)^{s/2}(|x|^{-b})=C_{d,b}|x|^{-b-s}$.}
\end{remark}

Using Lemma \ref{lem 2.7.}, Corollary \ref{cor 2.11.} and Remark \ref{rem 3.1.}, we have the following estimates of the term $|x|^{-b}|u|^{\sigma}u$
\begin{lemma}\label{lem 3.2.}
Let $d\in \mathbb N$, $0\le s<\frac{d}{2}$ , $0\le b<\min \{2, d-s,1+\frac{d}{2}-s \}$ and $0<\sigma\le\frac{4-2b}{d-2s}$. If $\sigma$ is not an even integer, assume further $\sigma>\left\lceil s\right\rceil-1$.
Then, for any interval $I(\subset \mathbb R)$, there exist $(\tilde{p},\tilde{q})\in S_0$ and $(\tilde{\alpha},\tilde{\beta})\in S_0$ such that
\begin{equation} \label{GrindEQ__3_1_}
\left\||x|^{-b}|u|^{\sigma}u\right\|_{L^{\tilde{\alpha}'}(I,\dot{H}_{\tilde{\beta}',2}^{s})}
\lesssim |I|^{\theta}\left\|u\right\|^{\sigma+1}_{L^{\tilde{p}}(I,\dot{H}_{\tilde{q},2}^{s})},
\end{equation}
\begin{equation} \label{GrindEQ__3_2_}
\left\||x|^{-b}|u|^{\sigma}v\right\|_{L^{\tilde{\alpha}'}(I,L^{\tilde{\beta}',2})}
\lesssim |I|^{\theta}\left\|u\right\|^{\sigma}_{L^{\tilde{p}}(I,\dot{H}_{\tilde{q},2}^{s})}\left\|u\right\|_{L^{p}(I,L^{\tilde{q},2})},
\end{equation}
where $\theta=\frac{(4-2b)-(d-2s)\sigma}{4}$.
\end{lemma}
\begin{proof}
We claim that there exist $(\tilde{\alpha},\tilde{\beta})\in S_0$ and $(\tilde{p},\tilde{q})\in S_0$ satisfying
\begin{equation} \label{GrindEQ__3_3_}
\frac{1}{\tilde{\beta}'} =\sigma \left(\frac{1}{\tilde{q}} -\frac{s}{d} \right)+\frac{1}{\tilde{q}}+\frac{b}{d}, ~\frac{1}{\tilde{q}} -\frac{s}{d} >0.
\end{equation}
In fact, we can choose $(p,q)\in S_0$ satisfying
$$
\max\left\{\frac{d-2}{2d},\;\frac{s}{d}\right\}<\frac{1}{\tilde{q}}<\frac{1}{2},
$$
provided that
\begin{equation}\label{GrindEQ__3_4_}
\max\left\{\frac{s}{d}+\frac{b}{d},\frac{(d-2)(\sigma+1)}{2d}-\frac{\sigma s}{d}+\frac{b}{d}\right\}<
\frac{1}{\tilde{\beta}'}<\frac{\sigma+1}{2}-\frac{\sigma s}{d}+\frac{b}{d}.
\end{equation}
And we can easily check that there exists $(\tilde{\alpha},\tilde{\beta})\in S_0$ satisfying \eqref{GrindEQ__3_4_} by using the fact $b<\min\{d-s,1+\frac{d}{2}-s\}$.

Let us consider the case $b>0$.
Using \eqref{GrindEQ__3_3_}, Lemmas \ref{lem 2.1.}, \ref{lem 2.4.}, \ref{lem 2.7.}, Corollary \ref{cor 2.11.} and Remark \ref{rem 3.1.}, we have
\begin{eqnarray}\begin{split} \label{GrindEQ__3_5_}
\left\| |x|^{-b}|u|^{\sigma}u\right\| _{\dot{H}_{\tilde{\beta}',2}^{s} } &\lesssim \left\| |x|^{-b}\right\| _{L^{q_{1},\infty}}\left\| |u|^{\sigma}u\right\| _{\dot{H}_{q_{2},2}^{s}}
+\left\| |x|^{-b}\right\| _{\dot{H}_{q_{3},\infty}^{s}}\left\| |u|^{\sigma}u\right\| _{L^{q_{4},2}}\\
&\lesssim \left\| |u|^{\sigma}u\right\| _{\dot{H}_{q_{2},2}^{s}}+\left\| |u|^{\sigma}u\right\| _{L^{q_{4},2}}\\
&\lesssim \left\| u\right\| _{L^{r,2(\sigma+1)}}^{\sigma}\left\| u\right\| _{\dot{H}_{\tilde{q},2(\sigma+1)}^{s}}+\left\| u\right\| _{L^{\tilde{\alpha},2(\sigma+1)}}^{\sigma+1}\\
&\lesssim \left\| u\right\| _{L^{r,2}}^{\sigma}\left\| u\right\| _{\dot{H}_{\tilde{q},2}^{s}}+\left\| u\right\| _{L^{\tilde{r},2}}^{\sigma+1}
\lesssim \left\| u\right\| _{\dot{H}_{\tilde{q},2}^{s} }^{\sigma +1}.
\end{split}\end{eqnarray}
where
\begin{equation}\nonumber
\frac{1}{q_{1}}:=\frac{b}{d},~\frac{1}{q_{2}}:=\frac{1}{\tilde{\beta}'}-\frac{b}{d},~\frac{1}{q_{3}}:=\frac{b+s}{d},~\frac{1}{q_{4}}:=\frac{1}{\tilde{\beta}'}-\frac{b+s}{d}~\frac{1}{\tilde{r}}:=\frac{1}{\tilde{q}}-\frac{s}{d}.
\end{equation}
Similarly, we also have
\begin{eqnarray}\begin{split} \label{GrindEQ__3_6_}
\left\| |x|^{-b}|u|^{\sigma}v\right\| _{L^{\tilde{\beta}',2}} &\lesssim  \left\| |x|^{-b}\right\| _{L^{q_{1},\infty}}\left\| |u|^{\sigma}v\right\| _{L^{q_{2},2}}
\lesssim \left\| u\right\| _{L^{\tilde{r},2}}^{\sigma}\left\| v\right\| _{L^{\tilde{q},2}}
\lesssim \left\| u\right\| _{\dot{H}_{\tilde{q},2}^{s} }^{\sigma }\left\| v\right\| _{L^{\tilde{q},2}}.
\end{split}\end{eqnarray}

Let us consider the case $b=0$. It follows from \eqref{GrindEQ__3_3_}, Lemma \ref{lem 2.4.} and Corollary \ref{cor 2.11.} that
\begin{eqnarray}\begin{split} \label{GrindEQ__3_7_}
\left\||u|^{\sigma}u\right\| _{\dot{H}_{\tilde{\beta}',2}^{s} }
\lesssim \left\| u\right\| _{L^{\tilde{r},2(\sigma+1)}}^{\sigma}\left\| u\right\| _{\dot{H}_{\tilde{q},2(\sigma+1)}^{s}}
\lesssim \left\| u\right\| _{L^{\tilde{r},2}}^{\sigma}\left\| u\right\| _{\dot{H}_{\tilde{q},2}^{s}}
\lesssim \left\| u\right\| _{\dot{H}_{\tilde{q},2}^{s} }^{\sigma +1}
\end{split}\end{eqnarray}
and
\begin{eqnarray}\begin{split} \label{GrindEQ__3_8_}
\left\||u|^{\sigma}u\right\| _{L^{\tilde{\beta}',2}}
\lesssim \left\| u\right\| _{L^{\tilde{r},2(\sigma+1)}}^{\sigma}\left\| u\right\| _{L^{\tilde{q},2(\sigma+1)}}
\lesssim \left\| u\right\| _{L^{\tilde{r},2}}^{\sigma}\left\| u\right\| _{L^{\tilde{q},2}}
\lesssim \left\| u\right\| _{\dot{H}_{\tilde{q},2}^{s} }^{\sigma}\left\| u\right\| _{L^{\tilde{q},2}}.
\end{split}\end{eqnarray}
On the other hand, \eqref{GrindEQ__3_3_} imply that
\begin{equation} \label{GrindEQ__3_9_}
\theta:=\frac{1}{\tilde{\alpha}'}-\frac{\sigma+1}{\tilde{p}}=\frac{(4-2b)-(d-2s)\sigma}{4}.
\end{equation}
Using \eqref{GrindEQ__3_5_}--\eqref{GrindEQ__3_9_} and H\"{o}lder's inequality, we immediately get \eqref{GrindEQ__3_1_} and \eqref{GrindEQ__3_2_}.
\end{proof}

\begin{lemma}\label{lem 3.3.}
Let $d\in \mathbb N$, $0\le s<\frac{d}{2}$ and $0<a<\min \{2, d-s,1+\frac{d}{2}-s \}$.
Then, for any interval $I(\subset \mathbb R)$, there exist $(\bar{p},\bar{q})\in S_0$ and $(\bar{\alpha},\bar{\beta})\in S_0$ such that
\begin{equation} \label{GrindEQ__3_10_}
\left\||x|^{-a}u\right\|_{L^{\bar{\alpha}'}(I,\dot{H}_{\bar{\beta}',2}^{s})}
\lesssim |I|^{\frac{2-a}{2}}\left\|u\right\|_{L^{\bar{p}}(I,\dot{H}_{\bar{q},2}^{s})},
\end{equation}
\begin{equation} \label{GrindEQ__3_11_}
\left\||x|^{-a}u\right\|_{L^{\bar{\alpha}'}(I,L^{\bar{\beta}',2})}
\lesssim |I|^{\frac{2-a}{2}}\left\|u\right\|_{L^{\bar{p}}(I,L^{\bar{q},2})}.
\end{equation}
\end{lemma}
\begin{proof}
We claim that there exist $(\bar{\alpha},\bar{\beta})\in S_0$ and $(\bar{p},\bar{q})\in S_0$ satisfying
\begin{equation} \label{GrindEQ__3_12_}
\frac{1}{\bar{\beta}'} =\frac{1}{\bar{q}}+\frac{a}{d}, ~\frac{1}{\bar{q}} -\frac{s}{d}=:\frac{1}{\bar{r}} >0.
\end{equation}
In fact, we can choose $\bar{q}$ satisfying
$$
\max\left\{\frac{d-2}{2d},\;\frac{s}{d}\right\}<\frac{1}{\bar{q}}<\frac{1}{2},
$$
provided that we choose $(\bar{\alpha},\bar{\beta})\in S_0$ satisfying
\begin{equation}\label{GrindEQ__3_13_}
\max\left\{\frac{d-2}{2d}+\frac{a}{d},\;\frac{s+a}{d}\right\}<
\frac{1}{\bar{\beta}'}<\frac{1}{2}+\frac{a}{d}.
\end{equation}
Using the fact $0<a<\min \{2, d-s,1+\frac{d}{2}-s \}$, we can easily check that there exists $(\bar{\alpha},\bar{\beta})\in S_0$ satisfying \eqref{GrindEQ__3_13_}.
Using \eqref{GrindEQ__3_12_}, Lemma \ref{lem 2.4.}, Lemma \ref{lem 2.7.}, Remark \ref{rem 3.1.} and H\"{o}lder's inequality, we have
\begin{eqnarray}\begin{split} \label{GrindEQ__3_14_}
\left\| |x|^{-a}u\right\| _{\dot{H}_{\bar{\beta}',2}^{s} } &\lesssim \left\| |x|^{-a}\right\| _{L^{{d}/a,\infty}}\left\| u\right\| _{\dot{H}_{\bar{q},2}^{s}}
+\left\| |x|^{-a}\right\| _{\dot{H}_{{d}/{(a+s)},\infty}^{s}}\left\| u\right\| _{L^{\bar{r},2}}
\lesssim \left\| u\right\| _{\dot{H}_{\bar{q},2}^{s} }
\end{split}\end{eqnarray}
and
\begin{eqnarray}\begin{split} \label{GrindEQ__3_15_}
\left\| |x|^{-a}u\right\| _{L^{\bar{\beta}',2}} &\lesssim  \left\| |x|^{-a}\right\| _{L^{d/a,\infty}}\left\| u\right\| _{L^{\bar{q},2}}
\lesssim \left\| u\right\| _{L^{\bar{q},2}}.
\end{split}\end{eqnarray}
Hence, using \eqref{GrindEQ__3_14_}, \eqref{GrindEQ__3_15_}, H\"{o}lder's inequality and the fact that $\frac{1}{\bar{\alpha}'}-\frac{1}{\bar{p}}=\frac{2-a}{2}$, we immediately get \eqref{GrindEQ__3_10_} and \eqref{GrindEQ__3_11_}.
\end{proof}

We are ready to prove Theorem \ref{thm 1.2.}.

\begin{proof}[{\bf Proof of Theorem \ref{thm 1.2.}}]
Let $T>0$ and $M>0$, which will be chosen later. Given $I=[-T,\;T]$, we define
\begin{equation} \nonumber
D=\left\{u\in L^{\tilde{p}} (I,H^{s}_{\tilde{q},2})\cap L^{\bar{p}} (I,H^{s}_{\bar{q},2}):~\left\| u\right\| _{L^{\tilde{p}} (I,H^{s}_{\tilde{q},2})}+\left\| u\right\| _{L^{\bar{p}} (I,H^{s}_{\bar{q},2})} \le M\right\},
\end{equation}
where $(\tilde{p},\tilde{q})\in S_0$ and $(\bar{p},\bar{q})\in S_0$ are as in Lemmas \ref{lem 3.2.} and \ref{lem 3.3.}.
Putting
\begin{equation} \nonumber
d(u,v)=\left\| u-v\right\| _{L^{\tilde{p}} (I,L^{\tilde{q},2})}+\left\| u-v\right\| _{L^{\bar{p}} (I,L^{\bar{q},2})},
\end{equation}
$(D,d)$ is a complete metric space (see e.g. \cite{AT21}).
Now we consider the mapping
\begin{equation} \label{GrindEQ__3_16_}
G:\;u(t)\to S(t)u_{0} -i \int _{0}^{t}S(t-\tau)(c|x|^{-a}u(\tau)+\lambda|x|^{-b} |u(\tau)|^{\sigma}u(\tau))d\tau.
\end{equation}
Using Lemma \ref{lem 2.13.} (Strichartz estimates) and Lemmas \ref{lem 3.2.}, \ref{lem 3.3.}, we have
\begin{eqnarray}\begin{split}\label{GrindEQ__3_17_}
&\left\|Gu \right\| _{L^{\tilde{p}} (I,H^{s}_{\tilde{q},2})\cap L^{\bar{p}} (I,H^{s}_{\bar{q},2})}\\
&~~~~~\lesssim\left\|S(t)u_{0}\right\|_{L^{\tilde{p}} (I,H^{s}_{\tilde{q},2})\cap L^{\bar{p}} (I,H^{s}_{\bar{q},2})}
+\left\| |x|^{-b}|u|^{\sigma}u\right\|_{L^{\tilde{\alpha}'} (I,H^{s}_{\tilde{\beta}',2})}+|c|\left\| |x|^{-a}u\right\|_{L^{\bar{\alpha}'} (I,H^{s}_{\bar{\beta}',2})}\\
&~~~~~\lesssim \left\|S(t)u_{0}\right\|_{L^{\tilde{p}} (I,H^{s}_{\tilde{q},2})\cap L^{\bar{p}} (I,H^{s}_{\bar{q},2})}+|I|^{\frac{(4-2b)-(d-2s)\sigma}{4}}\left\|u\right\|^{\sigma+1}_{L^{\tilde{p}}(I,H_{\tilde{q},2}^{s})}
+|c||I|^{\frac{2-a}{2}}\left\|u\right\|_{L^{\bar{p}}(I,H_{\bar{q},2}^{s})}
\end{split}\end{eqnarray}
and
\begin{eqnarray}\begin{split}\label{GrindEQ__3_18_}
d(Gu,Gv)&\lesssim\left\| |x|^{-b} \left(|u|^{\sigma}u-|v|^{\sigma}v\right)\right\|_{{L^{\tilde{\alpha}'} (I,L^{\tilde{\beta}',2})}}+|c|\left\| |x|^{-a} (u-v)\right\|_{{L^{\bar{\alpha}'} (I,L^{\bar{\beta}',2})}}\\
&\lesssim |I|^{\frac{(4-2b)-(d-2s)\sigma}{4}}\left\|u\right\|^{\sigma}_{L^{\tilde{p}}(I,H_{\tilde{q},2}^{s})}\left\|u-v\right\|_{L^{\tilde{p}}(I,L^{\tilde{q},2})}
+|c||I|^{\frac{2-a}{2}}\left\|u-v\right\|_{L^{\bar{p}}(I,L^{\bar{q},2})}
\end{split}\end{eqnarray}

First, we consider the $H^s$-subcritical case $\sigma<\frac{4-2b}{d-2s}$. Using \eqref{GrindEQ__3_17_}, \eqref{GrindEQ__3_18_} and Lemma \ref{lem 2.13.} (Strichartz estimates), we have
\begin{eqnarray}\begin{split}\label{GrindEQ__3_19_}
\left\|Gu \right\| _{L^{\tilde{p}} (I,H^{s}_{\tilde{q},2})\cap L^{\bar{p}} (I,H^{s}_{\bar{q},2})}
\lesssim \left\|u_{0}\right\|_{H^{s}}+|I|^{\frac{(4-2b)-(d-2s)\sigma}{4}}M^{\sigma+1}
+|c||I|^{\frac{2-a}{2}}M
\end{split}\end{eqnarray}
and
\begin{eqnarray}\begin{split}\label{GrindEQ__3_20_}
d(Gu,Gv)\lesssim (|I|^{\frac{(4-2b)-(d-2s)\sigma}{4}}M^{\sigma}+|c||I|^{\frac{2-a}{2}})d(u,v).
\end{split}\end{eqnarray}
Taking $M=2C\left\|u_{0}\right\|_{H^{s}}$ and choosing $T>0$ small enough so that
\begin{equation} \label{GrindEQ__3_21_}
C|I|^{\frac{(4-2b)-(d-2s)\sigma}{4}}M^{\sigma}+C|c||I|^{\frac{2-a}{2}}<\frac{1}{2},
\end{equation}
it follows from \eqref{GrindEQ__3_19_} and \eqref{GrindEQ__3_20_} that $G$ is a contraction on $(D,d)$.
The continuous dependence result follows from the above argument and Lemma \ref{lem 2.12.}, whose proof will be omitted.

Next, we consider the $H^s$-critical case $\sigma=\frac{4-2b}{d-2s}$. Using \eqref{GrindEQ__3_17_} and \eqref{GrindEQ__3_18_}, we have
\begin{eqnarray}\begin{split}\label{GrindEQ__3_22_}
\left\|Gu \right\| _{L^{\tilde{p}} (I,H^{s}_{\tilde{q},2})\cap L^{\bar{p}} (I,H^{s}_{\bar{q},2})}
\lesssim \left\|S(t)u_{0}\right\|_{L^{\tilde{p}} (I,H^{s}_{\tilde{q},2})\cap L^{\bar{p}} (I,H^{s}_{\bar{q},2})}+M^{\sigma+1}
+|c||I|^{\frac{2-a}{2}}M
\end{split}\end{eqnarray}
and
\begin{equation}\label{GrindEQ__3_23_}
d(Gu,Gv)\lesssim (M^{\sigma}+|c||I|^{\frac{2-a}{2}})d(u,v).
\end{equation}
By the Strichartz estimates \eqref{GrindEQ__2_3_}, we can also see that
\begin{equation}\nonumber
 \left\|S(t)u_{0}\right\|_{L^{\tilde{p}} (I,H^{s}_{\tilde{q},2})\cap L^{\bar{p}} (I,H^{s}_{\bar{q},2})}\to 0~\textnormal{as}~T\to 0
\end{equation}
Hence, taking $M>0$ such that $CM^{\sigma}\le \frac{1}{4}$ and $T>0$ such that
\begin{equation} \label{GrindEQ__3_24_}
\left\|S(t)u_{0}\right\|_{L^{\tilde{p}} (I,H^{s}_{\tilde{q},2})\cap L^{\bar{p}} (I,H^{s}_{\bar{q},2})}\le \frac{M}{4}~\textnormal{and}
~C|c||I|^{\frac{2-a}{2}}<\frac{1}{4},
\end{equation}
\eqref{GrindEQ__3_21_} and \eqref{GrindEQ__3_22_} imply that $G$ is a contraction on $(D,d)$. This completes the proof.
\end{proof}
\section{Standard continuous dependence in $H^s$}\label{sec 4.}
In this section, we prove Theorem \ref{thm 1.4.}.
To this end, we need the following estimates of the term $f(u)-f(v)$, where $f(u)$ is a nonlinear function that behaves like $|u|^{\sigma}u$.

\begin{lemma}[\cite{AK232}]\label{lem 4.1.}
Let $p>1$, $0<s<1$ and $\sigma > 1$. Assume that $f\in C^{2} \left(\mathbb C\to \mathbb C\right)$ satisfies
\begin{equation} \label{GrindEQ__4_1_}
|f^{(k)} (u)|\lesssim|u|^{\sigma +1-k} ,
\end{equation}
for any $0\le k\le 2$ and $u\in \mathbb C$.
Suppose also that
\begin{equation} \label{GrindEQ__4_2_}
\frac{1}{p} =\sigma \left(\frac{1}{r} -\frac{s}{d} \right)+\frac{1}{r} ,~\frac{1}{r} -\frac{s}{d} >0.
\end{equation}
Then we have
\begin{equation}\label{GrindEQ__4_3_}
\left\| f(u)-f(v)\right\| _{\dot{H}_{p,2}^{s} } \lesssim(\left\| u\right\| _{\dot{H}_{r,2}^{s} }^{\sigma } +\left\| v\right\| _{\dot{H}_{r,2}^{s} }^{\sigma })\left\| u-v\right\| _{\dot{H}_{r,2}^{s} } .
\end{equation}
\end{lemma}

\begin{lemma}\label{lem 4.2.}
Let $p>1$, $s\ge 1$ and $\sigma >\left\lceil s\right\rceil -1$.
Assume that $f\in C^{\left\lceil s\right\rceil } \left(\mathbb C\to \mathbb C\right)$ satisfies \eqref{GrindEQ__4_1_} for any $0\le k\le \left\lceil s\right\rceil $ and $u\in \mathbb C$. Assume further that
\begin{equation} \label{GrindEQ__4_4_}
|f^{\left(\left\lceil s\right\rceil\right)} (u)-f^{\left(\left\lceil s\right\rceil \right)} (v)|
\lesssim\left|u-v\right|^{\min \{ \sigma -\left\lceil s\right\rceil +1,1\} } \left(|u|+|v|\right)^{\max \{ 0,\sigma -\left\lceil s\right\rceil \} } ,
\end{equation}
for any $u,\;v\in \mathbb C$. Suppose also that \eqref{GrindEQ__4_2_} holds. Then we have
\begin{eqnarray}\begin{split} \label{GrindEQ__4_5_}
\left\| f(u)-f(v)\right\| _{\dot{H}_{p,2}^{s} }&\lesssim\left\| u-v\right\| _{L^{\gamma,2}}^{\min \{ \sigma -\left\lceil s\right\rceil +1,1\}}
(\left\| u\right\| _{\dot{H}_{r,2}^{s} }^{\max \{\left\lceil s\right\rceil,\sigma \}}+\left\|v\right\|_{\dot{H}_{r,2}^{s}}^{\max\{\left\lceil s\right\rceil,\sigma \}})\\
&~~+(\left\| u\right\| _{\dot{H}_{r,2}^{s} }^{\sigma } +\left\| v\right\| _{\dot{H}_{r,2}^{s} }^{\sigma })\left\| u-v\right\| _{\dot{H}_{r,2}^{s} } ,
\end{split}\end{eqnarray}
where $\gamma=\frac{rd}{d-rs} $. Moreover, if $\sigma \ge \left\lceil s\right\rceil $, then we have
\begin{equation} \label{GrindEQ__4_6_}
\left\| f(u)-f(v)\right\| _{\dot{H}_{p,2}^{s} } \lesssim(\left\| u\right\| _{\dot{H}_{r,2}^{s} }^{\sigma } +\left\| v\right\| _{\dot{H}_{r,2}^{s} }^{\sigma } )\left\| u-v\right\| _{\dot{H}_{r,2}^{s} } .
\end{equation}
\end{lemma}
\begin{proof}
The case $\sigma \ge \left\lceil s\right\rceil $ was proved in \cite{AK232}. However, we can easily extend this result to $\sigma>\left\lceil s\right\rceil-1$ by combining the arguments of Lemma 4.2 in \cite{AK232} and Lemma 3.2 in \cite{AKC22}, whose proof will be omitted.
\end{proof}

Similarly, we also have the following result.
\begin{lemma}[\cite{AK232}]\label{lem 4.3.}
Let $s>0$ and $f(z)$ be a polynomial in $z$ and $\bar{z}$ satisfying $\deg(f)=\sigma+1$. Suppose also that \eqref{GrindEQ__4_2_} holds.
Then we have \eqref{GrindEQ__4_6_}.
\end{lemma}

\begin{remark}\label{rem 4.4.}
\textnormal{Let $s>0$ and $\sigma>0$. If $\sigma $ is not an even integer, assume that $\sigma >\left\lceil s\right\rceil -1$. If $s<1$, in addition, suppose further that $\sigma \ge 1$. Then one can easily see that $f(u)=|u|^{\sigma } u$ satisfies the conditions of Lemmas \ref{lem 4.2.} and \ref{lem 4.3.}. See \cite{AKC22,CFH11} for example.}
\end{remark}

Using Remark \ref{rem 3.1.} and Lemmas \ref{lem 4.1.}--\ref{lem 4.3.}, we have the following estimates of the term $|x|^{-b}(|u|^{\sigma}u-|v|^{\sigma}v)$.
\begin{lemma}\label{lem 4.5.}
Let $1<p,\;r<\infty $, $b>0$, $s>0$, $b+s<d$ and $\sigma>0$. If $\sigma $ is not an even integer, assume that $\sigma >\left\lceil s\right\rceil -1$. If $s<1$, in addition, suppose further that $\sigma >1$.
Suppose also that
\begin{equation} \label{GrindEQ__4_7_}
\frac{1}{p} =\sigma \left(\frac{1}{r} -\frac{s}{d} \right)+\frac{1}{r}+\frac{b}{d}, ~\frac{1}{r} -\frac{s}{d} >0.
\end{equation}
Then we have
\begin{eqnarray}\begin{split} \label{GrindEQ__4_8_}
\left\| |x|^{-b}(|u|^{\sigma}u-|v|^{\sigma}v)\right\| _{\dot{H}_{p,2}^{s} }&\lesssim\left\| u-v\right\| _{L^{\gamma,2}}^{\min \{ \sigma -\left\lceil s\right\rceil +1,1\}}
(\left\| u\right\| _{\dot{H}_{r,2}^{s} }^{\max \{\left\lceil s\right\rceil,\sigma \}}+\left\|v\right\|_{\dot{H}_{r,2}^{s}}^{\max\{\left\lceil s\right\rceil,\sigma \}})\\
&~~+(\left\| u\right\| _{\dot{H}_{r,2}^{s} }^{\sigma } +\left\| v\right\| _{\dot{H}_{r,2}^{s} }^{\sigma })\left\| u-v\right\| _{\dot{H}_{r,2}^{s} } ,
\end{split}\end{eqnarray}
where $\gamma=\frac{rd}{d-rs} $. Moreover, if $\sigma \ge \left\lceil s\right\rceil $, then we have
\begin{equation} \label{GrindEQ__4_9_}
\left\||x|^{-b}(|u|^{\sigma}u-|v|^{\sigma}v)\right\| _{\dot{H}_{p,2}^{s} } \lesssim(\left\| u\right\| _{\dot{H}_{r,2}^{s} }^{\sigma } +\left\| v\right\| _{\dot{H}_{r,2}^{s} }^{\sigma } )\left\| u-v\right\| _{\dot{H}_{r,2}^{s} } .
\end{equation}
\end{lemma}
\begin{proof}
Using Lemma \ref{lem 2.7.} (fractional product rule), Lemmas \ref{lem 2.4.}, \ref{lem 4.1.}--\ref{lem 4.3.}, Remark \ref{rem 3.1.} and H\"{o}lder's inequality in Lorentz spaces, we immediately get
\begin{eqnarray}\begin{split} \nonumber
&\left\| |x|^{-b}(|u|^{\sigma}u-|v|^{\sigma}v)\right\| _{\dot{H}_{p,2}^{s} }\\
&~~~~~\lesssim\left\| |x|^{-b}\right\|_{L^{d/b,\infty}}\left\| |u|^{\sigma}u-|v|^{\sigma}v\right\| _{\dot{H}_{p_1,2}^{s}}+\left\| |x|^{-b}\right\|_{H_{d/(b+s),\infty}^{s}}\left\| |u|^{\sigma}u-|v|^{\sigma}v\right\| _{L^{p_2,2}}\\
&~~~~~\lesssim \left\| |u|^{\sigma}u-|v|^{\sigma}v\right\| _{\dot{H}_{p_1,2}^{s}}+\left\| (|u|^{\sigma}+|v|^{\sigma})(u-v)\right\| _{L^{p_2,2}}\\
&~~~~~\lesssim\left\| u-v\right\| _{L^{\gamma,2}}^{\min \{ \sigma -\left\lceil s\right\rceil +1,1\}}
(\left\| u\right\| _{\dot{H}_{r,2}^{s} }^{\max \{\left\lceil s\right\rceil,\sigma \}}+\left\|v\right\|_{\dot{H}_{r,2}^{s}}^{\max\{\left\lceil s\right\rceil,\sigma \}})\\
&~~~~~~~~~~+(\left\| u\right\| _{\dot{H}_{r,2}^{s} }^{\sigma } +\left\| v\right\| _{\dot{H}_{r,2}^{s} }^{\sigma })\left\| u-v\right\| _{\dot{H}_{r,2}^{s} }
+(\left\|u\right\|_{L^{\gamma,2}}^{\sigma}+\left\|v\right\|_{L^{\gamma,2}}^{\sigma})\left\|u-v\right\|_{L^{\gamma,2}}\\
&~~~~~\lesssim\left\| u-v\right\| _{L^{\gamma,2}}^{\min \{ \sigma -\left\lceil s\right\rceil +1,1\}}
(\left\| u\right\| _{\dot{H}_{r,2}^{s} }^{\max \{\left\lceil s\right\rceil,\sigma \}}+\left\|v\right\|_{\dot{H}_{r,2}^{s}}^{\max\{\left\lceil s\right\rceil,\sigma \}})\\
&~~~~~~~~~~+(\left\| u\right\| _{\dot{H}_{r,2}^{s} }^{\sigma } +\left\| v\right\| _{\dot{H}_{r,2}^{s} }^{\sigma })\left\| u-v\right\| _{\dot{H}_{r,2}^{s} },
\end{split}\end{eqnarray}
where
$$
\frac{1}{p_1}:=\sigma \left(\frac{1}{r} -\frac{s}{d} \right)+\frac{1}{r},~\frac{1}{p_2}:=(\sigma+1) \left(\frac{1}{r} -\frac{s}{d} \right).
$$
This completes the proof of \eqref{GrindEQ__4_8_}.
If $\sigma \ge \left\lceil s\right\rceil$, then \eqref{GrindEQ__4_9_} follows directly from \eqref{GrindEQ__4_8_} by using the embedding $\dot{H}_{r,2}^{s}\hookrightarrow L^{\gamma,2}$.
\end{proof}

\begin{lemma}\label{lem 4.6.}
Let $d\in \mathbb N$, $0\le s<\frac{d}{2}$ , $0\le b<\min \{2, d-s,1+\frac{d}{2}-s \}$ and $0<\sigma\le\frac{4-2b}{d-2s}$.
If $\sigma$ is not an even integer, assume further $\sigma>\left\lceil s\right\rceil-1$.
If $s<1$, in addition, suppose further that $\sigma >1$. Let $I(\subset \mathbb R)$ be an interval and $\theta=\frac{(4-2b)-(d-2s)\sigma}{4}$.
\begin{enumerate}
  \item If either $\sigma$ is an even integer or $\sigma\ge \left\lceil s\right\rceil$, then we have
  \begin{equation} \label{GrindEQ__4_10_}
  \left\||x|^{-b}(|u|^{\sigma}u-|v|^{\sigma}v)\right\|_{L^{\tilde{\alpha}'}(I,\dot{H}_{\tilde{\beta}',2}^{s})}
  \lesssim |I|^{\theta}(\left\|u\right\|^{\sigma}_{L^{\tilde{p}}(I,\dot{H}_{\tilde{q},2}^{s})}+\left\|v\right\|^{\sigma}_{L^{\tilde{p}}(I,\dot{H}_{\tilde{q},2}^{s})})
  \left\|u-v\right\|_{L^{\tilde{p}}(I,\dot{H}_{\tilde{q},2}^{s})},
  \end{equation}
  where $(\tilde{p},\tilde{q})\in S_0$ and $(\tilde{\alpha},\tilde{\beta})\in S_0$ are as in Lemma \ref{lem 3.2.}.
  \item If $\sigma$ is not an even integer and $\left\lceil s\right\rceil>\sigma>\left\lceil s\right\rceil-1$, then we have
  \begin{eqnarray}\begin{split} \label{GrindEQ__4_11_}
  \left\||x|^{-b}(|u|^{\sigma}u-|v|^{\sigma}v)\right\|_{L^{\tilde{\alpha}'}(I,\dot{H}_{\tilde{\beta}',2}^{s})}
  &\lesssim |I|^{\theta}(\left\|u\right\|^{\sigma}_{L^{\tilde{p}}(I,\dot{H}_{\tilde{q},2}^{s})}+\left\|v\right\|^{\sigma}_{L^{\tilde{p}}(I,\dot{H}_{\tilde{q},2}^{s})})
  \left\|u-v\right\|_{L^{\tilde{p}}(I,\dot{H}_{\tilde{q},2}^{s})}\\
  &~~~~+(\left\|u\right\|^{\left\lceil s\right\rceil}_{L^{\tilde{p}}(I,\dot{H}_{\tilde{q},2}^{s})}+\left\|v\right\|^{\left\lceil s\right\rceil}_{L^{\tilde{p}}(I,\dot{H}_{\tilde{q},2}^{s})})
  \left\|u-v\right\|_{L^{\hat{p}}(I,L^{\tilde{r},2})}^{\sigma+1-\left\lceil s\right\rceil},
  \end{split}\end{eqnarray}
  where $(\tilde{p},\tilde{q})\in S_0$ and $(\tilde{\alpha},\tilde{\beta})\in S_0$ are as in Lemma \ref{lem 3.2.} and
  \begin{equation} \label{GrindEQ__4_12_}
  \frac{1}{\tilde{r}}=\frac{1}{\tilde{q}}-\frac{s}{d},~\frac{1}{\hat{p}}=\frac{1}{\tilde{p}}+\frac{\theta}{\sigma+1-\lceil s\rceil}.
  \end{equation}
\end{enumerate}
\end{lemma}
\begin{proof}
Using \eqref{GrindEQ__3_3_}, \eqref{GrindEQ__3_9_}, \eqref{GrindEQ__4_12_}, Lemma \ref{lem 4.5.} and H\"{o}lder's inequality, we immediately get \eqref{GrindEQ__4_10_} and \eqref{GrindEQ__4_11_}.
\end{proof}

\begin{proof}[{\bf Proof of Theorem \ref{thm 1.4.}}]
Since $u$, $u_{n} $ satisfy the following integral equations:
\[u(t)=S(t)u_{0}-i\int_{0}^{t}S(t-\tau)(c|x|^{-a}u(\tau)+\lambda|x|^{-b} |u(\tau)|^{\sigma}u(\tau))d\tau  ,\]
\[u_{n} (t)=S(t)u_{0,n} -i \int _{0}^{t}S(t-\tau)(c|x|^{-a}u_n(\tau)+\lambda |x|^{-b} |u_n(\tau)|^{\sigma}u_n(\tau))d\tau  ,\]
respectively, we have
\begin{eqnarray}\begin{split} \nonumber
u_{n} (t)-u(t)&=S(t) \left(u_{0,n} -u_{0} \right)-ic \int _{0}^{t}S(t-\tau) |x|^{-a} \left(|u_n(\tau)|-u(\tau)\right)d\tau\\
&-i\lambda \int _{0}^{t}S(t-\tau) |x|^{-b} \left(|u_n(\tau)|^{\sigma}u_n(\tau)-|u(\tau)|^{\sigma}u(\tau)\right)d\tau.
\end{split}\end{eqnarray}
We are to prove that there exits $T>0$ sufficiently small such that as $n\to \infty $,
\begin{equation} \label{GrindEQ__4_13_}
u_{n} \to u ~~\textrm{in}~~ L^{p} ([-T,T],\, H_{q,2}^{s}(\mathbb R^{d})),
\end{equation}
for every admissible pair $(p,q)$.
If this has been done, then the result follows by iterating this property to cover any compact subset of $(-T_{\min } ,T_{\max })$ in the $H^s$-subcritical case (see e.g. Chapter 3 or 4 of \cite{C03}) or a standard compact argument in the $H^s$-critical case (see e.g. Subsection 3.2 of \cite{DYC13}).
We divide the proof of \eqref{GrindEQ__4_13_} in two cases: $H^s$-subcritical case and $H^s$-critical case.

\textbf{Case 1.} We consider the $H^s$-subcritical case $\sigma<\frac{4-2b}{d-2s}$.
Since $u_{0,n}\to u_{0}$ in $H^s$, we have
$$
\left\|u_{0,n}\right\|_{H^s}\le 2\left\|u_{0}\right\|_{H^s},
$$
for $n$ large enough.
Using the argument in the proof of Theorem \ref{thm 1.2.}, we can construct solutions $u$ and $u_{n}$ ($n$ sufficiently large) in the set $(D,d)$ given in Theorem \ref{thm 1.2.}, which implies that there exist $T>0$ such that $T<T_{\max}(u_{0}),T_{\min}(u_{0})$ and $T<T_{\max}(u_{0,n}),T_{\min}(u_{0,n})$ for $n$ large enough. Furthermore, for all sufficiently large $n$, we have
$$
\left\|u_{n}\right\|_{L^{\tilde{p}} (I,H^{s}_{\tilde{q},2})\cap L^{\bar{p}} (I,H^{s}_{\bar{q},2})}\le M,
$$
where $I=[-T,T]$ and $(\tilde{p},\tilde{q})\in S_0$, $(\bar{p},\bar{q})\in S_0$, $M>0$ are as in the proof of Theorem \ref{thm 1.2.}.

\textit{Case 1.1.} We consider the case $\sigma\ge \lceil s\rceil$ or the case that $\sigma$ is an even integer. Using Lemmas \ref{lem 2.13.}, \ref{lem 3.2.}, \ref{lem 3.3.}, \ref{lem 4.6.} and \eqref{GrindEQ__3_21_}, we have
\begin{eqnarray}\begin{split}\label{GrindEQ__4_14_}
\left\|u_n-u\right\|_{L^{\tilde{p}} (I,H^{s}_{\tilde{q},2})\cap L^{\bar{p}} (I,H^{s}_{\bar{q},2})}&\lesssim
\left\|u_{0,n}-u_0\right\|_{H^{s}}+
\left\| |x|^{-b} \left(|u_n|^{\sigma}u_n-|u|^{\sigma}u\right)\right\|_{{L^{\tilde{\alpha}'} (I,H_{\tilde{\beta}',2}^s)}}\\
&~~~+|c|\left\| |x|^{-a} (u_n-u)\right\|_{{L^{\bar{\alpha}'} (I,H_{\bar{\beta}',2}^s)}}\\
&\lesssim |I|^{\frac{(4-2b)-(d-2s)\sigma}{4}}(\left\|u_n\right\|^{\sigma}_{L^{\tilde{p}}(I,H_{\tilde{q},2}^{s})}+\left\|u\right\|^{\sigma}_{L^{\tilde{p}}(I,H_{\tilde{q},2}^{s})})
\left\|u_n-u\right\|_{L^{\tilde{p}}(I,H_{\tilde{q},2}^s)}\\
&~~~+|c||I|^{\frac{2-a}{2}}\left\|u_n-u\right\|_{L^{\bar{p}}(I,H_{\bar{q},2}^s)}+\left\|u_{0,n}-u_0\right\|_{H^{s}}\\
&\le (2C|I|^{\frac{(4-2b)-(d-2s)\sigma}{4}}M^{\sigma}+C|c||I|^{\frac{2-a}{2}})\left\|u_n-u\right\|_{L^{\tilde{p}} (I,H^{s}_{\tilde{q},2})\cap L^{\bar{p}} (I,H^{s}_{\bar{q},2})}\\
&~~~+\left\|u_{0,n}-u_0\right\|_{H^{s}}\\
&\le \frac{1}{2}\left\|u_n-u\right\|_{L^{\tilde{p}} (I,H^{s}_{\tilde{q},2})\cap L^{\bar{p}} (I,H^{s}_{\bar{q},2})}+\left\|u_{0,n}-u_0\right\|_{H^{s}},
\end{split}\end{eqnarray}
which shows that
\begin{equation}\label{GrindEQ__4_15_}
\left\|u_n-u\right\|_{L^{\tilde{p}} (I,H^{s}_{\tilde{q},2})\cap L^{\bar{p}} (I,H^{s}_{\bar{q},2})}\to 0,~ \textnormal{as}~n\to \infty.
\end{equation}
Using Lemma \ref{lem 2.13.} (Strichartz estimates) and \eqref{GrindEQ__4_15_} and repeating the same argument, we have, for any admissible pair $(p,q)$,
\begin{eqnarray}\begin{split}\nonumber
\left\|u_n-u\right\|_{L^{p} (I,H^{s}_{q,2})}&\lesssim
\left\|u_{0,n}-u_0\right\|_{H^{s}}+\left\| |x|^{-b} \left(|u_n|^{\sigma}u_n-|u|^{\sigma}u\right)\right\|_{{L^{\tilde{\alpha}'} (I,H_{\tilde{\beta}',2}^s)}}\\
&~~~+|c|\left\| |x|^{-a} (u_n-u)\right\|_{{L^{\bar{\alpha}'} (I,H_{\bar{\beta}',2}^s)}}\\
&\le \frac{1}{2}\left\|u_n-u\right\|_{L^{\tilde{p}} (I,H^{s}_{\tilde{q},2})\cap L^{\bar{p}} (I,H^{s}_{\bar{q},2})}+\left\|u_{0,n}-u_0\right\|_{H^{s}}\to 0
~\textnormal{as}~n\to \infty.
\end{split}\end{eqnarray}

\textit{Case 1.2.} We consider the case that $\sigma$ is not an even integer and $\lceil s\rceil>\sigma> \lceil s\rceil-1$.
Repeating the same argument as in the proof of \eqref{GrindEQ__4_14_}, we have
\begin{eqnarray}\begin{split}\label{GrindEQ__4_16_}
\left\|u_n-u\right\|_{L^{\tilde{p}} (I,H^{s}_{\tilde{q},2})\cap L^{\bar{p}} (I,H^{s}_{\bar{q},2})}
&\le \frac{1}{2}\left\|u_n-u\right\|_{L^{\tilde{p}} (I,H^{s}_{\tilde{q},2})\cap L^{\bar{p}} (I,H^{s}_{\bar{q},2})}+\left\|u_{0,n}-u_0\right\|_{H^{s}}\\
&~~~~+(\left\|u_n\right\|^{\left\lceil s\right\rceil}_{L^{\tilde{p}}(I,\dot{H}_{\tilde{q},2}^{s})}+\left\|u\right\|^{\left\lceil s\right\rceil}_{L^{\tilde{p}}(I,\dot{H}_{\tilde{q},2}^{s})})
  \left\|u_n-u\right\|_{L^{\hat{p}}(I,L^{\tilde{r},2})}^{\sigma+1-\left\lceil s\right\rceil},
\end{split}\end{eqnarray}
where $\hat{p}$ and $\tilde{r}$ are given in \eqref{GrindEQ__4_12_}. Since $(\tilde{p},\tilde{q})\in S_0$, we can take $\eta >0$ sufficiently small such that
\begin{equation} \label{GrindEQ__4_17_}
\tilde{q}_1:=\tilde{q}+\eta,~(\tilde{p}_1,\tilde{q}_1)\in S_0,~\hat{p}<\tilde{p}_1<\tilde{p},~s-\frac{\eta d}{\tilde{q} \tilde{q}_1} >0.
\end{equation}
By \eqref{GrindEQ__4_17_} and Lemma \ref{lem 2.4.}, there holds the embedding $\dot{H}_{\tilde{q}_1}^{s-\frac{\eta d}{\tilde{q} \tilde{q}_1}}\hookrightarrow L^{\tilde{r}}$.
Hence, it follows from \eqref{GrindEQ__4_17_}, H\"{o}lder's inequality and Theorem \ref{thm 1.2.} that
\begin{equation} \label{GrindEQ__4_18_}
\left\|u_n-u\right\|_{L^{\hat{p}}(I,L^{\tilde{r},2})}\le (2T)^{\frac{1}{\hat{p}}-\frac{1}{\tilde{p}_1}}\left\|u_n-u\right\|_{L^{\tilde{p}_1}(I,H_{\tilde{q}_1,2}^{s-\frac{\eta d}{\tilde{q} \tilde{q}_1}})} \to 0,~\textnormal{as}~n\to \infty.
\end{equation}
Using \eqref{GrindEQ__4_16_}, \eqref{GrindEQ__4_18_} and the same argument as in Case 1.1, we can get the desired result.

\textbf{Case 2.} Next, we consider the $H^s$-critical case $\sigma=\frac{4-2b}{d-2s}$.
Since $u_{0,n}\to u_{0}$ in $H^s$, it follows from Lemma \ref{lem 2.13.} (Strichartz estimates) that
\begin{equation}\label{GrindEQ__4_19_}
\left\|S(t) u_{0,n}\right\|_{L^{p}([-T,T], H_{q,2}^{s})}\le 2\left\|S(t) u_{0}\right\|_{L^{p}([-T,\;T], H_{q,2}^{s})}.
\end{equation}
for any admissible pair $(p,q)$ and $n$ large enough. Hence, by using the argument similar to that in Case 1.1, we can get the desired result.
\end{proof}

\section{Global existence of $H^1$-solution}\label{sec 5.}
In this section, we prove Theorems \ref{thm 1.6.}.
\begin{proof}[{\bf Proof of Theorem \ref{thm 1.6.}}]
Thanks to the local well-posedness in Proposition \ref{prp 1.1.} and the conservation of mass, the global existence follows if we can show that there exists $C>0$ independent of $t$ such that
$$
\left\|\nabla u(t)\right\|_{L^2}\le C,$$
for any $t$ in the existence time.

{\bf Proof of Item 1.} We consider the defocusing case $\lambda=1$.
If $c\ge0$, then it follows from the conservation of energy (see \eqref{GrindEQ__1_20_} for the definition of energy) that
$$
\left\|\nabla u(t)\right\|_{L^2}\le \sqrt{2E_{b,c}(u_0)},
$$
for any $t$ in the existence time. Hence, it suffices to consider the case $c<0$.
By the definition of energy \eqref{GrindEQ__1_20_}, we have
\begin{eqnarray}\begin{split}\label{GrindEQ__5_1_}
E_{b,c}(u(t))&=\frac{1}{2}\left\|\nabla u(t)\right\|_{L^2}^2-\frac{|c|}{2}\left\||x|^{-a} |u(t)|^2\right\|_{L^1}+\frac{1}{\sigma+2}\left\||x|^{-b}|u(t)|^{\sigma+2}\right\|_{L^1}\\
&\ge \frac{1}{2}\left\|\nabla u(t)\right\|_{L^2}^2-\frac{|c|}{2}\left\||x|^{-a} |u(t)|\right\|_{L^1}.
\end{split}\end{eqnarray}
On the other hand, it follows from Corollary \ref{cor 2.7.}, Lemma \ref{lem 2.12.} and Young's inequality
\footnote[2]{\ Let $a$, $b$ be non-negative real numbers and $p$, $q$ be positive real numbers satisfying $\frac{1}{p}+\frac{1}{q}=1$. Then for any $\varepsilon$, we have $ab\lesssim\varepsilon a^{p}+\varepsilon^{-\frac{q}{p}}b^{q}$.}
\begin{eqnarray}\begin{split}\label{GrindEQ__5_2_}
\left\||x|^{-a} |u(t)|\right\|_{L^1}&=\left\||x|^{-\frac{a}{2}} |u(t)|\right\|_{L^2}^2\le C\left\|u(t)\right\|_{\dot{H}_2^{\frac{a}{2}}}^2
\le C\left\|\nabla u(t)\right\|_{L^2}^{a}\left\|u(t)\right\|_{L^2}^{2-a}\\
&\le \frac{1}{2|c|}\left\|\nabla u(t)\right\|_{L^2}^2+C(a,|c|)\left\|u(t)\right\|_{L^2}^{2}.
\end{split}\end{eqnarray}
\eqref{GrindEQ__5_1_}, \eqref{GrindEQ__5_2_} and the conservation of mass and energy imply that
we have
$$
\left\|\nabla u(t)\right\|_{L^2}^2 \le 4E_{b,c}(u_0)+C(a,|c|)M(u_0),
$$
for any $t$ in the existence time. This completes the proof of Item 1.

It remains to consider the focusing case $\lambda=-1$.

{\bf Proof of Item 2.} We consider the mass-subcritical case $\sigma<\frac{4-2b}{d}$.
If $c\ge0$, then it follows from Lemma \ref{lem 2.14.} (Gagliardo-Nirenberg inequality) that
\begin{eqnarray}\begin{split}\label{GrindEQ__5_3_}
E_{b,c}(u(t))&=\frac{1}{2}\left\|\nabla u(t)\right\|_{L^2}^2+\frac{c}{2}\left\||x|^{-a} |u(t)|^2\right\|_{L^1}-\frac{1}{\sigma+2}\left\||x|^{-b}|u(t)|^{\sigma+2}\right\|_{L^1}\\
&\ge \frac{1}{2}\left\|\nabla u(t)\right\|_{L^2}^2-\frac{C_{\rm GN}}{\sigma+2}\left\|\nabla u(t)\right\|_{L^2}^{\frac{d\sigma+2b}{2}}
\left\|u(t)\right\|_{L^{2}}^{\frac{4-2b-\sigma(d-2)}{2}}.
\end{split}\end{eqnarray}
Since $\sigma<\frac{4-2b}{d}$, we use Young's inequality and the conservation of mass to get that
\begin{equation}\label{GrindEQ__5_4_}
\frac{C_{\rm GN}}{\sigma+2}\left\|\nabla u(t)\right\|_{L^2}^{\frac{d\sigma+2b}{2}}
\left\|u(t)\right\|_{L^{2}}^{\frac{4-2b-\sigma(d-2)}{2}}\le \varepsilon \left\|\nabla u(t)\right\|_{L^2}^{2}+C(\varepsilon, M(u_0)),
\end{equation}
for any $\varepsilon>0$. \eqref{GrindEQ__5_3_}, \eqref{GrindEQ__5_4_} and the conservation of energy imply that
$$
\left(\frac{1}{2}-\varepsilon\right)\left\|\nabla u(t)\right\|_{L^2}^2\le E_{b,c}(u_0)+C(\varepsilon, M(u_0)).
$$
Taking $\varepsilon\in (0,\frac{1}{2})$, we get the uniform bound on $\left\|\nabla u(t)\right\|_{L^2}$.

If $c<0$, then it follows from  Lemma \ref{lem 2.14.} (Gagliardo-Nirenberg inequality), \eqref{GrindEQ__5_4_} and \eqref{GrindEQ__5_2_} that
\begin{eqnarray}\begin{split}\label{GrindEQ__5_5_}
E_{b,c}(u(t))&=\frac{1}{2}\left\|\nabla u(t)\right\|_{L^2}^2-\frac{|c|}{2}\left\||x|^{-a} |u(t)|^2\right\|_{L^1}-\frac{1}{\sigma+2}\left\||x|^{-b}|u(t)|^{\sigma+2}\right\|_{L^1}\\
&\ge \left(\frac{1}{4}-\varepsilon\right)\left\|\nabla u(t)\right\|_{L^2}^2-C(a,|c|)M(u_0)-C(\varepsilon, M(u_0)).
\end{split}\end{eqnarray}
Taking $\varepsilon\in (0,\frac{1}{4})$, the desired result follows from \eqref{GrindEQ__5_5_} and the conservation of energy.

{\bf Proof of Item 3.} We consider the case mass-critical case $\sigma=\frac{4-2b}{d}$.

Let us consider the case $c\ge0$. It follows from Lemma \ref{lem 2.14.} (Gagliardo-Nirenberg inequality) with \eqref{GrindEQ__2_7_} and the conservation of mass that
\begin{eqnarray}\begin{split}\label{GrindEQ__5_6_}
E_{b,c}(u(t))&=\frac{1}{2}\left\|\nabla u(t)\right\|_{L^2}^2+\frac{c}{2}\left\||x|^{-a} |u(t)|^2\right\|_{L^1}-\frac{d}{2d+4-2b}\left\||x|^{-b}|u(t)|^{\frac{4-2b}{d}+2}\right\|_{L^1}\\
&\ge \frac{1}{2}\left\|\nabla u(t)\right\|_{L^2}^2-\frac{1}{2}\left(\frac{\left\|u(t)\right\|_{L^{2}}}{\left\|Q\right\|_{L^{2}}}\right)^{\frac{4-2b}{d}}\left\|\nabla u(t)\right\|_{L^2}^{2}\\
&=\frac{1}{2}\left(1-\left(\frac{\left\|u(t)\right\|_{L^{2}}}{\left\|Q\right\|_{L^{2}}}\right)\right)\left\|\nabla u(t)\right\|_{L^2}^{2}.
\end{split}\end{eqnarray}
Since $\left\|u_0\right\|_{L^2}<\left\|Q\right\|_{L^2}$, \eqref{GrindEQ__5_6_} and the conservation of energy show that the corresponding solutin $u$ exists globally in time.

Let us consider the case $c<0$. Using the same argument as in the proof of \eqref{GrindEQ__5_2_}, we also have
\begin{eqnarray}\begin{split}\nonumber
\left\||x|^{-a} |u(t)|\right\|_{L^1}\le \frac{\varepsilon}{|c|}\left\|\nabla u(t)\right\|_{L^2}^2+C(a,|c|,\varepsilon)\left\|u(t)\right\|_{L^2}^{2},
\end{split}\end{eqnarray}
for any $\varepsilon>0$. Using this inequality, Lemma \ref{lem 2.14.} (Gagliardo-Nirenberg inequality) and the conservation of mass, we have
\begin{eqnarray}\begin{split}\label{GrindEQ__5_7_}
E_{b,c}(u(t))&=\frac{1}{2}\left\|\nabla u(t)\right\|_{L^2}^2-\frac{|c|}{2}\left\||x|^{-a} |u(t)|^2\right\|_{L^1}-\frac{d}{2d+4-2b}\left\||x|^{-b}|u(t)|^{\frac{4-2b}{d}+2}\right\|_{L^1}\\
&\ge \frac{1}{2}\left\|\nabla u(t)\right\|_{L^2}^2-\frac{\varepsilon}{2}\left\|\nabla u(t)\right\|_{L^2}^2-C(a,|c|,\varepsilon)\left\|u(t)\right\|_{L^2}^{2}\\
&~~~~-\frac{1}{2}\left(\frac{\left\|u(t)\right\|_{L^{2}}}{\left\|Q\right\|_{L^{2}}}\right)^{\frac{4-2b}{d}}\left\|\nabla u(t)\right\|_{L^2}^{2}\\
&=\frac{1}{2}\left(1-\varepsilon-\left(\frac{\left\|u(t)\right\|_{L^{2}}}{\left\|Q\right\|_{L^{2}}}\right)\right)\left\|\nabla u(t)\right\|_{L^2}^{2}-C(a,|c|,\varepsilon)M(u_0).
\end{split}\end{eqnarray}
Since $\left\|u_0\right\|_{L^2}<\left\|Q\right\|_{L^2}$,  we can get the desired result from \eqref{GrindEQ__5_7_} by taking $\varepsilon>0$ sufficiently small.

{\bf Proof of Item 4.} Let us consider the intercritical case $\frac{4-2b}{d}<\sigma<\sigma_{\rm c}(1,b)$. Since $c\ge0$, it follows from Lemma \ref{lem 2.14.} (Gagliardo-Nirenberg inequality) that
\begin{eqnarray}\begin{split}\label{GrindEQ__5_8_}
&E_{b,c}(u(t))[M(u(t))]^{\gamma_{\rm c}}\\
&~~~~~~~~~=\frac{1}{2}\left\|\nabla u(t)\right\|_{L^2}^2\left\|u(t)\right\|_{L^2}^{2\gamma_{\rm c}}-\frac{|c|}{2}\left\||x|^{-a} |u(t)|^2\right\|_{L^1}\left\|u(t)\right\|_{L^2}^{2\gamma_{\rm c}}\\
&~~~~~~~~~~~~~~~-\frac{1}{\sigma+2}\left\||x|^{-b}|u(t)|^{\sigma+2}\right\|_{L^1}\left\|u(t)\right\|_{L^2}^{2\gamma_{\rm c}}\\
&~~~~~~~~~\ge \frac{1}{2}\left(\left\|\nabla u(t)\right\|_{L^2}\left\|u(t)\right\|_{L^2}^{\gamma_{\rm c}}\right)^2
-\frac{C_{\rm GN}}{\sigma+2}\left\|\nabla u(t)\right\|_{L^2}^{\frac{d\sigma+2b}{2}}
\left\|u(t)\right\|_{L^{2}}^{\frac{4-2b-\sigma(d-2)}{2}+2\gamma_c}\\
&~~~~~~~~~=f\left(\left\|\nabla u(t)\right\|_{L^2}\left\|u(t)\right\|_{L^2}^{\gamma_{\rm c}}\right),
\end{split}\end{eqnarray}
where $f(x):=\frac{1}{2}x^2-\frac{C_{\rm GN}}{\sigma+2}x^{\frac{d\sigma+2b}{2}}$.
On the other hand, \eqref{GrindEQ__2_6_} and \eqref{GrindEQ__2_8_} show that
\begin{equation}\label{GrindEQ__5_9_}
f\left(\left\|\nabla Q\right\|_{L^2}\left\|Q\right\|_{L^2}^{\gamma_{\rm c}}\right)=E_{b}(Q)[M(Q)]^{\gamma_c}.
\end{equation}
Using \eqref{GrindEQ__5_8_}, \eqref{GrindEQ__5_9_}, \eqref{GrindEQ__1_22_} and the conservation of mass and energy, we get
\begin{equation}\label{GrindEQ__5_10_}
f\left(\left\|\nabla u(t)\right\|_{L^2}\left\|u(t)\right\|_{L^2}^{\gamma_{\rm c}}\right)<f\left(\left\|\nabla Q\right\|_{L^2}\left\|Q\right\|_{L^2}^{\gamma_{\rm c}}\right),
\end{equation}
for any $t$ in the existence time. \eqref{GrindEQ__5_10_}, the second condition in \eqref{GrindEQ__1_22_} and the continuity argument imply that
$$
\left\|\nabla u(t)\right\|_{L^2}\left\|u(t)\right\|_{L^2}^{\gamma_{\rm c}}<\left\|\nabla Q\right\|_{L^2}\left\|Q\right\|_{L^2}^{\gamma_{\rm c}},
$$
for any $t$ in the existence time. The result follows from the above inequality and the conservation of mass.
\end{proof}
\section{Blow-up of $H^1$-solution}\label{sec 6.}
In this section, we prove Theorems \ref{thm 1.7.} and \ref{thm 1.9.}.
\subsection{Virial estimates}
In this subsection, we derive the virial estimates related to \eqref{GrindEQ__1_1_} in the focusing case which are useful to study the finite or infinite time blow-up of $H^1$-solution.
To this end, we recall the Gagliardo-Nirenberg inequality. See Corollary 1.3 of \cite{WHHG11} for example.
\begin{lemma}[\cite{WHHG11}]\label{lem 6.1.}
Let $1<p,p_0,p_1<\infty$, $s,s_1\in \mathbb R$, $0\le \theta\le 1$. The the fractional GN inequality of the following type
$$
\left\|u\right\|_{\dot{H}_p^s}\lesssim \left\|u\right\|_{L^{p_0}}^{1-\theta}\left\|u\right\|_{\dot{H}_{p_1}^{s_1}}^{\theta}
$$
holds if and only if
$$
\frac{d}{p}-s=(1-\theta)\frac{d}{p_0}+\theta\left(\frac{d}{p_1}-s_1\right),~s\le \theta s_1.
$$
\end{lemma}

Given a real valued function $\omega$, we define the virial potential by
\begin{equation}\nonumber
V_{\omega}(t):=\int{\omega(x)\left|u(t, x)\right|^{2}dx.}
\end{equation}
A simple computation shows that the following result holds.
\begin{lemma}[\cite{D21A}]\label{lem 6.2.}
Let $V,W:\mathbb R^d\to \mathbb R$. If $u$ is a (sufficiently smooth and decaying) solution to
$iu_t+\Delta u=Vu+W|u|^{\sigma}u$,
then it holds that
\begin{equation}\label{GrindEQ__6_1_}
\frac{d(t)}{dt}V_{\omega}=2\int{\nabla \omega \cdot {\rm Im}(\bar{u}(t)\nabla u(t))dx}
\end{equation}
and
\begin{eqnarray}\begin{split}\label{GrindEQ__6_2_}
\frac{d^{2}}{dt^{2}}V_{\omega}(t)&=-\int{\Delta^{2}a(x)|u(t)|^{2}dx+4\sum_{j,k=1}^{d}
{\int{\partial_{jk}^{2}a(x)\textnormal{Re}(\partial_{k}u(t)\partial_{j}\bar{u}(t))dx}}}\\
&~~~-2\int{\nabla \omega \cdot\nabla V|u(t)|^{2}dx}+\frac{2\sigma}{\sigma+2}\int{\Delta \omega W|u(t)|^{\sigma+2}dx}\\
&~~~-\frac{4}{\sigma+2}\int{\nabla \omega \cdot \nabla W |u(t)|^{\sigma+2}dx}.
\end{split}\end{eqnarray}
\end{lemma}
Let us introduce a function $\varphi:[0,\infty)\to [0,\infty)$ satisfying
\begin{equation}\label{GrindEQ__6_3_}
\varphi (r)=\left\{\begin{array}{l} {r^2,~\textrm{if}\;0\le r\le 1,}
\\ {\textrm{0},~\textrm{if}\;r\ge 2,} \end{array}\right.
\textrm{and}~~\varphi''(r)\le 2 ~~\textrm{for}~~r\ge 0.
\end{equation}
Given $R>0$, we define the radial function $\varphi_{R}:\mathbb R^{d}\to [0,\infty)$:
\begin{equation}\label{GrindEQ__6_4_}
\varphi_{R}(x)=\varphi_{R}(r):=R^{2}\varphi(r/R),\;r=|x|.
\end{equation}
One can easily see that
\begin{equation}\label{GrindEQ__6_5_}
2-\varphi''_{R}(r)\ge0,\;2-\frac{\varphi'_{R}(r)}{r}\ge0,\;2d-\Delta\varphi_{R}(x)\ge0,~\forall r\ge 0,~x\in \mathbb R^d.
\end{equation}

We have the following localized virial estimate.
\begin{lemma}\label{lem 6.3.}
Let $d \in \mathbb N$, $0<a,b<2$, $c\in \mathbb R$, $0<\sigma<\infty$, $\sigma\le \sigma_{\rm c}(1,b)$, $R>1$ and $\varphi_{R}$ be as in \eqref{GrindEQ__6_4_}. Let $u:I\times\mathbb R^{d}\to \mathbb C$ be a solution to the focusing Cauchy problem \eqref{GrindEQ__1_1_} with $\lambda=-1$. Then for any $t\in I$, we have
\begin{eqnarray}\begin{split} \label{GrindEQ__6_6_}
\frac{d^2}{dt^2}V_{\varphi_{R}}(t)&\le 8 \left\|\nabla u(t)\right\|_{L^2}^{2}+4ac\left\||x|^{-a} |u(t)|^2\right\|_{L^1}-\frac{4(d\sigma+2b)}{\sigma+2}\int{|x|^{-b}|u(t)|^{\sigma+2}dx}\\
&~~~~~+CR^{-2}+CR^{-a}+CR^{-b}\left\|\nabla u(t)\right\|_{L^2}^{\frac{\sigma d}{2}}\\
&=G(u(t))-2c(d\sigma+2b-2a)\left\||x|^{-a}|u(t)|^2\right\|_{L^1}+CR^{-2}+CR^{-a}+CR^{-b}\left\|\nabla u(t)\right\|_{L^2}^{\frac{\sigma d}{2}},
\end{split}\end{eqnarray}
where
\begin{eqnarray}\begin{split}\label{GrindEQ__6_7_}
G(u):&=8\left\|\nabla u\right\|_{L^2}^2+2c(d\sigma+2b)\left\||x|^{-a}|u|^2\right\|_{L^1}-\frac{4(d\sigma+2b)}{\sigma+2}\int{|x|^{-b}|u|^{\sigma+2}dx}\\
&=4(d\sigma+2b)E_{b,c}(u)-2(d\sigma+2b-4)\left\|\nabla u\right\|_{L^2}^2.
\end{split}\end{eqnarray}

\end{lemma}
\begin{proof}
Applying Lemma \ref{lem 6.2.} with $V=c|x|^{-a}$ and $W=-|x|^{-b}$, we have
\begin{eqnarray}\begin{split}\label{GrindEQ__6_8_}
\frac{d^{2}}{dt^{2}}V_{\varphi_{R}}(t)&=-\int{\Delta^{2}\varphi_{R}(x)|u(t, x)|^{2}dx}-2c\int{\nabla \varphi_{R}(x)\cdot\nabla(|x|^{-a})|u(t, x)|^{2}dx}\\
&~~+4\sum_{j,k=1}^{d} {\int{\partial_{jk}^{2}\varphi_{R}(x)\textnormal{Re}(\partial_{k}u(t, x)\partial_{j}\bar{u}(t, x))dx}}\\
&~~-\frac{2\sigma}{\sigma+2}
\int{\Delta \varphi_{R}(x)|x|^{-b}|u(t, x)|^{\sigma+2}dx}\\
&~~+\frac{4}{\sigma+2}\int{\nabla \varphi_{R}(x)\cdot \nabla(|x|^{-b})|u(t, x)|^{\sigma+2}dx}.
\end{split}\end{eqnarray}
Noticing that
$$\partial_{j}=\frac{x_{j}}{r}\partial_{r},~\partial^{2}_{jk}=\left(\frac{\delta_{jk}}{r}-\frac{x_{j}x_{k}}{r^{3}}\right)\partial_{r}
+\frac{x_{j}x_{k}}{r^{2}}\partial^{2}_{r},$$
we have
\begin{eqnarray}\begin{split}\label{GrindEQ__6_9_}
&\sum_{j,k=1}^{d}
{\int{\partial_{jk}^{2}\varphi_{R}(x)\textnormal{Re}(\partial_{k}u(t, x)\partial_{j}\bar{u}(t, x))dx}} \\&~~~~~~=\int{\frac{\varphi'_{R}(r)}{r}\left|\nabla u(t, x)\right|^{2}dx}+\int{\left(\frac{\varphi''_{R}(r)}{r^{2}}-\frac{\varphi'_{R}(r)}{r^{3}}\right)
\left|x\cdot\nabla u(t, x)\right|^{2}dx.}
\end{split}\end{eqnarray}
and
\begin{equation}\label{GrindEQ__6_10_}
\nabla \varphi_{R}(x)\cdot \nabla(|x|^{-a})=-a\frac{\varphi_{R}'}{r}|x|^{-a},~\nabla \varphi_{R}(x)\cdot \nabla(|x|^{-b})=-b\frac{\varphi_{R}'}{r}|x|^{-b}.
\end{equation}
In view of \eqref{GrindEQ__6_8_}--\eqref{GrindEQ__6_10_}, we have
\begin{eqnarray}\begin{split}\label{GrindEQ__6_11_}
\frac{d^{2}}{dt^{2}}V_{\varphi_{R}}(t)&=-\int{\Delta^{2}\varphi_{R}(x)|u(t, x)|^{2}dx}+4\int{\frac{\varphi'_{R}(r)}{r}\left|\nabla u(t, x)\right|^{2}dx}\\
&+4\int{\left(\frac{\varphi''_{R}(r)}{r^{2}}-\frac{\varphi'_{R}(r)}{r^{3}}\right)\left|x\cdot\nabla u(t, x)\right|^{2}dx}
\\
&-2ac\int{\left(2-\frac{\varphi_{R}'}{r}\right)|x|^{-a}|u(t, x)|^{2}dx}+4ac\left\||x|^{-a} |u(t)|^2\right\|_{L^1}
\\
&-\frac{4b}{\sigma+2}\int{|x|^{-b}
\frac{\varphi'(r)}{r}|u(t, x)|^{\sigma+2}dx}\\
&-\frac{2\sigma}{\sigma+2}\int{\Delta \varphi_{R}(x)|x|^{-b}|u(t, x)|^{\sigma+2}dx}.
\end{split}\end{eqnarray}
Since $\left\|\Delta^{2}\varphi_{R}\right\|_{L^{\infty}}\lesssim R^{-2}$, the conservation of mass implies that
\begin{equation}\label{GrindEQ__6_12_}
\left|\int{\Delta^{2}\varphi_{R}(x)|u(t, x)|^{2}dx}\right|\lesssim R^{-2}\left\|u(t)\right\|_{L^2}^{2}\lesssim R^{-2}.
\end{equation}
Using the conservation of mass and the facts that
$
\left|x\cdot \nabla u\right|\le r|\nabla u|,~\varphi''_{R}(r)\le 2,~\frac{\varphi'_{R}(r)}{r}\le 2,
$
we get
\begin{eqnarray}\begin{split}\label{GrindEQ__6_13_}
&4\int{\frac{\varphi'_{R}(r)}{r}\left|\nabla u(t)\right|^{2}dx}
+4\int{\left(\frac{\varphi''_{R}(r)}{r^{2}}-\frac{\varphi'_{R}(r)}{r^{3}}\right)\left|x\cdot\nabla u(t)\right|^{2}dx}
-8\left\|\nabla u(t)\right\|_{L^2}\\
&~~~~\le 4 \int{\left(2-\frac{\varphi'_{R}(r)}{r}\right)\left(-\left|\nabla u(t)\right|^{2}+\frac{\left|x\cdot\nabla u(t)\right|^{2}}{r^2}\right)dx}\le 0.
\end{split}\end{eqnarray}
Using \eqref{GrindEQ__6_5_} and the conservation of mass, we also have
\begin{equation}\label{GrindEQ__6_14_}
-2ac\int{\left(2-\frac{\varphi_{R}'}{r}\right)|x|^{-a}|u(t, x)|^{2}dx}\lesssim
\max\left\{-2acS,0\right\} \int_{|x|>R}{R^{-a}|u(t)|^{2}dx}
\lesssim R^{-a},
\end{equation}
where $S=\max_{r\ge 1}\{2-\frac{\theta'(r)}{r}\}$.
Furthermore, it follows from Lemma \ref{lem 6.1.} and the conservation of mass that
\begin{eqnarray}\begin{split}\label{GrindEQ__6_15_}
\frac{4(d\sigma+2b)}{\sigma+2}&\int{|x|^{-b}|u(t, x)|^{\sigma+2}dx}-\frac{2\sigma}{\sigma+2}\int{\Delta \varphi_{R}(x)|x|^{-b}|u(t, x)|^{\sigma+2}dx}\\
&-\frac{4b}{\sigma+2}\int{|x|^{-b}\frac{\varphi'(r)}{r}|u(t, x)|^{\sigma+2}dx}\\
&=\frac{2\sigma}{\sigma+2}\int{|x|^{-b}\left(2d-\Delta \varphi_{R}(x)\right)|u(t)|^{\sigma+2}dx}\\
&+\frac{4b}{\sigma+2}\int{|x|^{-b}\left(2-\frac{\varphi'_{R}(r)}{r}\right)|u(t)|^{\sigma+2}dx}\\
&\lesssim\int_{|x|\ge R}{|x|^{-b}|u(t)|^{\sigma+2}}\le R^{-b}\left\|u(t)\right\|_{L^{\sigma+2}}^{\sigma+2}\\
&\lesssim R^{-b}\left\|\nabla u(t)\right\|_{L^2}^{\frac{\sigma d}{2}}\left\| u(t)\right\|_{L^2}^{\sigma+2-\frac{\sigma d}{2}}\lesssim R^{-b}\left\|\nabla u(t)\right\|_{L^2}^{\frac{\sigma d}{2}},
\end{split}\end{eqnarray}
where we use the fact $\sigma\le \sigma_{\rm c}(1,b)$.
The desired result follows from \eqref{GrindEQ__6_11_}--\eqref{GrindEQ__6_15_}.
\end{proof}
In the mass-critical case $\sigma=\frac{4-2b}{d}$, we have the following refined version of Lemma \ref{lem 6.2.}. As in \cite{CF22}, we use the following function
\begin{equation} \label{GrindEQ__6_16_}
v(r)=\left\{\begin{array}{ll}
{2r,}~&{\textrm{if}~0\le r\le 1,}\\
{2r-2(r-1)^k,}~&{\textrm{if}~1< r\le 1+\left(\frac{1}{k}\right)^{\frac{1}{k-1}},}\\
{\textrm{smooth and } v'=0,}~&{\textrm{if}~1+\left(\frac{1}{k}\right)^{\frac{1}{k-1}}< r<2,}\\
{0,}~&{\textrm{if}~r\ge 2,}
\end{array}\right.
\end{equation}
where $k\in \mathbb N$ is chosen as in \cite{CF22}.
We then define the radial function
$$
\phi(r)=\int_{0}^{r}v(s)ds.
$$
\begin{equation}\label{GrindEQ__6_17_}
\phi_{R}(x)=\phi_{R}(r):=R^{2}\phi(r/R),\;r=|x|.
\end{equation}
\begin{lemma}\label{lem 6.4.}
Let $d \in \mathbb N$, $0<a\le 2$, $0<b<2$, $c\in \mathbb R$, $\sigma=\frac{4-2b}{d}$ and $\phi_{R}$ be as in \eqref{GrindEQ__6_17_}. Let $u:I\times\mathbb R^{d}\to \mathbb C$ be a solution to the focusing Cauchy problem \eqref{GrindEQ__1_1_} with $\lambda=-1$ and $E_{b,c}(u_0)<0$. Then for any $t\in I$ and $R>1$ large enough, we have
\begin{eqnarray}\begin{split} \label{GrindEQ__6_18_}
\frac{d^2}{dt^2}V_{\phi_{R}}(t)&\le 8 E_{b,c}(u_0)-4c(2-a)\left\||x|^{-a}|u(t)|^2\right\|_{L^1}.
\end{split}\end{eqnarray}
\end{lemma}
\begin{proof}
Using the same argument as in the proof of Lemma \ref{lem 6.2.}, we immediately get
\begin{eqnarray}\begin{split}\label{GrindEQ__6_19_}
\frac{d^{2}}{dt^{2}}V_{\phi_{R}}(t)&=16E_{b,c}(u_0)-4c(2-a)\left\||x|^{-a}|u(t)|^2\right\|_{L^1}+K_1+K_2+K_2+K_4,
\end{split}\end{eqnarray}
where
$$
K_1=-4\int{\left(2-\frac{\phi'_{R}(r)}{r}\right)\left|\nabla u(t)\right|^{2}dx}
+4\int{\left(\frac{\phi''_{R}(r)}{r^{2}}-\frac{\phi'_{R}(r)}{r^{3}}\right)\left|x\cdot\nabla u(t)\right|^{2}dx},
$$
$$
K_2=\frac{2}{d+2-b}\int{\left[(2-b)(2-\phi''_{R}(r))+(2d-2+b)\left(2-\frac{\phi'_{R}(r)}{r}\right)\right]|x|^{-b}|u(t)|^{\sigma+2}dx},
$$
$$
K_3=-\int{\Delta^{2}\phi_{R}(x)|u(t, x)|^{2}dx}
$$
and
$$
K_4=-2ac\int{\left(2-\frac{\phi_{R}'}{r}\right)|x|^{-a}|u(t, x)|^{2}dx}.
$$
Using the same argument as in \eqref{GrindEQ__6_14_}, we get
\begin{equation}\label{GrindEQ__6_20_}
K_4 \lesssim R^{-a}.
\end{equation}
Using \eqref{GrindEQ__6_19_}, \eqref{GrindEQ__6_20_} and the estimates of $K_i~(i=1,2,3)$ obtained in \cite{CF22}, we immediately get \eqref{GrindEQ__6_18_} for $R>1$ large enough and we omitted the details.
\end{proof}
\subsection{Proofs of Theorems \ref{thm 1.7.} and \ref{thm 1.9.}}

We are ready to prove Theorems \ref{thm 1.7.} and \ref{thm 1.9.}.
\begin{proof}[{\bf Proof of Item 1 of Theorem \ref{thm 1.7.} and Theorem \ref{thm 1.9.}}]
Using Lemma \ref{lem 6.4.} and the assumptions of Item 1 of Theorem \ref{thm 1.7.} and Theorem \ref{thm 1.9.}, we immediately get
$$
\frac{d^2}{dt^2}V_{\phi_{R}}(t)\le 8 E_{b,c}(u_0)<0,
$$
where $\phi_{R}$ is given in \eqref{GrindEQ__6_17_}.
Integrating this estimate, there exists $t_0>0$ such that $V_{\phi_{R}}(t_0)<0$ which is impossible. This completes the proof.
\end{proof}
\begin{proof}[{\bf Proof of Items 2 and 3 of Theorem \ref{thm 1.7.}.}]
We prove the results in three steps.

\textbf{Step 1.}
In this step, we prove that there exists $c_1>0$ such that
\begin{equation}\label{GrindEQ__6_21_}
G(u(t))\le-c_1,
\end{equation}
for any $t$ in the existence time, where $G(u)$ is given in \eqref{GrindEQ__6_7_}.
In fact, if $E_{b,c}(u_0)<0$, we immediately get \eqref{GrindEQ__6_21_} by using \eqref{GrindEQ__6_7_}, the conservation of energy and fact that $\sigma>\frac{4-2b}{d}$. Hence, it suffices to consider the case $E_{b,c}(u_0)\ge0$.

First, we consider the intercritical case $\frac{4-2b}{d}<\sigma<\sigma_{\rm c}(1,b)$.
Using \eqref{GrindEQ__5_10_}, the second condition in \eqref{GrindEQ__1_23_} and the continuity argument, we have
\begin{equation}\label{GrindEQ__6_22_}
\left\|\nabla u(t)\right\|_{L^2}\left\|u(t)\right\|_{L^2}^{\gamma_{\rm c}}>\left\|\nabla Q\right\|_{L^2}\left\|Q\right\|_{L^2}^{\gamma_{\rm c}},
\end{equation}
for any $t$ in the existence time.
And the first condition in \eqref{GrindEQ__1_23_} shows that there exists $\delta>0$ such that
\begin{equation}\label{GrindEQ__6_23_}
E_{b,c}(u_{0})[M(u_{0})]^{\gamma_{\rm c}}\le (1-\delta)E_{b}(Q)[M(Q)]^{\gamma_{\rm c}}.
\end{equation}
\eqref{GrindEQ__6_22_}, \eqref{GrindEQ__6_23_}, \eqref{GrindEQ__2_6_} and the conservation of mass and energy show that
\begin{eqnarray}\begin{split}\label{GrindEQ__6_24_}
G(u(t))[M(u(t))]^{\gamma_{\rm c}}&=4(d\sigma+2b)E_{b,c}(u(t))[M(u(t))]^{\gamma_{\rm c}}\\
&~~~~~~~~~~-2(d\sigma-4+2b)
\left(\left\|\nabla u(t)\right\|_{L^2}\left\|u(t)\right\|_{L^{2}}^{\gamma_{\rm c}}\right)^2\\
&\le 4(d\sigma+2b)(1-\delta)E_{b}(Q)[M(Q)]^{\gamma_{\rm c}}\\
&~~~~~~~~~~-2(d\sigma-4+2b)
\left(\left\|\nabla Q\right\|_{L^2}\left\|Q\right\|_{L^{2}}^{\gamma_{\rm c}}\right)^2\\
&=-2(d\sigma-4+2b)\delta\left(\left\|\nabla Q\right\|_{L^2}\left\|Q\right\|_{L^{2}}^{\gamma_{\rm c}}\right)^2,
\end{split}\end{eqnarray}
for any $t$ in the existence time.
Hence we have \eqref{GrindEQ__6_21_} with
$$
c_1=\frac{2(d\sigma-4+2b)\delta\left(\left\|\nabla Q\right\|_{L^2}\left\|Q\right\|_{L^{2}}^{\gamma_{\rm c}}\right)^2}{[M(u_0)]^{\gamma_{\rm c}}}.
$$

Next, we consider the energy-critical case $\sigma=\frac{4-2b}{d-2}$ with $d\ge3$.
By definition of the energy and Lemma \ref{lem 2.16.} (sharp Hardy-Sobolev inequality), we have
\begin{eqnarray}\begin{split}\nonumber
E_{b,c}(u(t))&=\frac{1}{2} \left\|\nabla u(t)\right\| _{L^2}+\frac{c}{2}\left\||x|^{-a}|u(t)|^2\right\|_{L^1} -\frac{1}{\sigma+2} \int{|x|^{-b} \left|u(t,x)\right|^{\sigma+2}dx}\\
&\ge \frac{1}{2} \left\|\nabla u(t)\right\|_{L^2} -\frac{[C_{\rm HS}]^{\sigma+2}}{\sigma+2}\left\|\nabla u(t)\right\| _{L^2}^{\sigma+2}=:g\left(\left\|\nabla u(t)\right\| _{L^2}\right),
\end{split}\end{eqnarray}
where
\begin{equation}\label{GrindEQ__6_25_}
g(y)=\frac{1}{2}y^2-\frac{[C_{\rm HS}]^{\sigma+2}}{\sigma+2}y^{\sigma+2}.
\end{equation}
\eqref{GrindEQ__2_14_} shows that
\[g\left(\left\|\nabla W_{b}\right\|_{L^2}\right)=E_b(W_{b}).\]
By the conservation of energy and the first condition in \eqref{GrindEQ__1_24_}, we can see that
\[g\left(\left\|\nabla u(t)\right\| _{L^2}\right)\le E_{b,c}(u(t))=E_{b,c}\left(u_{0}\right)<E_b(W_{b}).\]
By the second condition in \eqref{GrindEQ__1_24_} and the continuity argument, we have
\begin{equation}\label{GrindEQ__6_26_}
\left\|\nabla u(t)\right\|_{L^2}>\left\|\nabla W_{b}\right\|_{L^2},
\end{equation}
for any $t$ in the existence time. By the second condition in \eqref{GrindEQ__1_24_}, we can take $\delta>0$ small enough such that
\begin{equation}\label{GrindEQ__6_27_}
E_{b,c}(u_{0})\le(1-\delta)E_{b}(W_{b}).
\end{equation}
\eqref{GrindEQ__2_13_}, \eqref{GrindEQ__2_14_}, \eqref{GrindEQ__6_7_}, \eqref{GrindEQ__6_26_} and \eqref{GrindEQ__6_27_} show that
\begin{eqnarray}\begin{split} \nonumber
G(u(t))&= 4(d\sigma+2b)E_{b,c}(u(t))-2(d\sigma+2b-4)\left\|\nabla u\right\|_{L^2}^2\\
&\le4(1-\delta)(d\sigma+2b)E_{b}(W_{b})-2(d\sigma+2b-4)\left\|\nabla W_b\right\|_{L^2}^2\\
&=-\frac{ 8\delta(2-b)}{d-2}\left\|\nabla W_{b}\right\|_{L^2}^{2},
\end{split}\end{eqnarray}
which completes the proof of \eqref{GrindEQ__6_21_}.

\textbf{Step 2.}
In this step, we prove the first parts of Items 2 and 3 in Theorem \ref{thm 1.6.}. If $T^{*}<\infty$, then we are done. If $T^{*}=\infty$, then we have to show that there exists $t_{n}\to \infty$ such that $\left\|\nabla u(t_{n})\right\|_{L^2}\to \infty$ as $n\to \infty$. Assume by contradiction that it doesn't hold, i.e.
\begin{equation}\label{GrindEQ__6_28_}
\sup_{t\in [0,\infty)}{\left\|\nabla u(t)\right\|_{L^2}}\le M_{0},
\end{equation}
for some $M_{0}>0$.
Using Lemma \ref{lem 6.3.}, \eqref{GrindEQ__6_21_}, \eqref{GrindEQ__6_28_} and the fact $d\sigma+2b>4>2a$, we can take $R>1$ large enough such that
\begin{eqnarray}\begin{split} \label{GrindEQ__6_29_}
\frac{d^2}{dt^2}V_{\varphi_{R}}(t)&\le -c_1-2c(d\sigma+2b-2a)\left\||x|^{-a}|u(t)|^2\right\|_{L^1}+CR^{-2}+CR^{-a}+CR^{-b}M_0^{\frac{\sigma d}{2}}\\
&=-\frac{c_1}{2},
\end{split}\end{eqnarray}
for any $t$ in the existence time, where $\varphi_{R}$ is given in \eqref{GrindEQ__6_4_}. This shows that there exists $t_1>0$ such that $V_{\varphi_{R}}(t_1)<0$ which is impossible. This completes the proof of the first parts in Items 2 and 3 of Theorem \ref{thm 1.6.}.

\textbf{Step 3.}
In this step, we prove the second parts of Items 2 and 3 in Theorem \ref{thm 1.6.}.
By assumption, we have $\sigma<\frac{4}{d}$ in both of intercritical and energy-critical case.

We first claim that there exists $c_2>0$ such that
\begin{equation}\label{GrindEQ__6_30_}
\left\|\nabla u(t)\right\|_{L^2}\ge c_2,
\end{equation}
for all $t\in[0,T_{\max})$. In fact, \eqref{GrindEQ__6_7_}, \eqref{GrindEQ__6_21_} and the fact $c\ge 0$ imply that
\begin{equation}\label{GrindEQ__6_31_}
\left\|\nabla u\right\|_{L^2}^2\lesssim \int{|x|^{-b}|u|^{\sigma+2}dx}.
\end{equation}
On the other hand, the conservation of mass, Lemmas \ref{lem 2.14.} and \ref{lem 2.16.} show that
\begin{equation}\label{GrindEQ__6_32_}
\int{|x|^{-b}|u|^{\sigma+2}dx}\lesssim \left\|\nabla u\right\|_{L^2}^{\frac{d\sigma+2b}{2}}.
\end{equation}
\eqref{GrindEQ__6_30_} follows directly from \eqref{GrindEQ__6_31_}, \eqref{GrindEQ__6_32_} and the fact $\sigma>\frac{4-2b}{d}$.

We next claim that there exists $c_3>0$ such that
\begin{equation}\label{GrindEQ__6_33_}
\frac{d^2}{dt^2}V_{\varphi_{R}}(t)\le -c_3\left\|\nabla u(t)\right\|_{L^2}^2,
\end{equation}
for all $t\in[0,T_{\max})$ and $R>1$ large enough, where $\varphi_R$ is as in \eqref{GrindEQ__6_4_}. To prove \eqref{GrindEQ__6_33_}, we denote
\begin{equation}\label{GrindEQ__6_34_}
\gamma:=\frac{4(d\sigma+2b)|E_{b,c}(u_0)|+4}{d\sigma+2b-4}.
\end{equation}

\emph{Case 1.} We consider the case
\begin{equation}\label{GrindEQ__6_35_}
\left\|\nabla u(t)\right\|_{L^2}^2\le \gamma.
\end{equation}
Lemma \ref{lem 6.3.}, \eqref{GrindEQ__6_21_}, \eqref{GrindEQ__6_35_} and the fact that $c\ge 0$, we immediately get
$$
\frac{d^2}{dt^2}V_{\varphi_{R}}(t)\le -c_1+CR^{-2}+CR^{-a}+CR^{-b}\gamma^{\frac{\sigma d}{2}},
$$
By choosing $R>1$ large enough, it follows from \eqref{GrindEQ__6_35_} that
\begin{equation}\label{GrindEQ__6_36_}
\frac{d^2}{dt^2}V_{\varphi_{R}}(t)\le -\frac{c_1}{2}\le -\frac{c_1}{2\gamma}\left\|\nabla u(t)\right\|_{L^2}^2.
\end{equation}

\emph{Case 2.} We consider the case
\begin{equation}\label{GrindEQ__6_37_}
\left\|\nabla u(t)\right\|_{L^2}^2\ge \gamma.
\end{equation}
In this case, we have
\begin{eqnarray}\begin{split} \label{GrindEQ__6_38_}
G(u(t))&= 4(d\sigma+2b)E_{b,c}(u(t))-2(d\sigma+2b-4)\left\|\nabla u\right\|_{L^2}^2\\
&\le 4(d\sigma+2b)E_{b,c}(u(t))-(d\sigma+2b-4)\gamma-(d\sigma+2b-4)\left\|\nabla u\right\|_{L^2}^2\\
&\le -4-(d\sigma+2b-4)\left\|\nabla u\right\|_{L^2}^2
\end{split}\end{eqnarray}
Since $c\ge 0$ and $\sigma<\frac{4}{d}$, it follows from Lemma \ref{lem 6.3.}, \eqref{GrindEQ__6_30_} and \eqref{GrindEQ__6_38_} that
\begin{eqnarray}\begin{split} \label{GrindEQ__6_39_}
\frac{d^2}{dt^2}V_{\varphi_{R}}(t)&\le -4-(d\sigma+2b-4)\left\|\nabla u\right\|_{L^2}^2+CR^{-2}+CR^{-a}+CR^{-b}\left\|\nabla u(t)\right\|_{L^2}^{\frac{\sigma d}{2}}\\
&\le-4-(d\sigma+2b-4-CR^{-b}c_2^{\frac{\sigma d}{2}-2})\left\|\nabla u\right\|_{L^2}^2+CR^{-2}+CR^{-a},
\end{split}\end{eqnarray}
Due to the fact that $\sigma>\frac{4-2b}{d}$, we can choose $R>1$ large enough such that
\begin{equation}\label{GrindEQ__6_40_}
CR^{-b}c_2^{\frac{\sigma d}{2}-2}\le \frac{d\sigma+2b-4}{2}~\textrm{and} ~CR^{-2}+CR^{-a}\le 4.
\end{equation}
\eqref{GrindEQ__6_39_} and \eqref{GrindEQ__6_40_} imply that
\begin{equation} \label{GrindEQ__6_41_}
\frac{d^2}{dt^2}V_{\varphi_{R}}(t)\le -\frac{d\sigma+2b-4}{2}\left\|\nabla u\right\|_{L^2}^2,
\end{equation}
for $R>1$ large enough.
In both of Case 1 and Case 2, the choice of $R>1$ are independent of $t$. Hence, we get \eqref{GrindEQ__6_33_} with
$c_3:=\min\{\frac{c_1}{2\gamma},\frac{d\sigma+2b-4}{2}\}>0$. This complete the proof of \eqref{GrindEQ__6_33_}.

By time integration, \eqref{GrindEQ__6_30_} and \eqref{GrindEQ__6_33_} imply that
\begin{equation} \label{GrindEQ__6_42_}
V_{\varphi_{R}}'(t)\lesssim -t<0,~\forall t>T,
\end{equation}
for $T>0$ sufficiently large.
By time integration again, \eqref{GrindEQ__6_33_} and \eqref{GrindEQ__6_42_} show that
\begin{equation} \label{GrindEQ__6_43_}
V_{\varphi_{R}}'(t)\lesssim -\int_{T}^{t}{\left\|\nabla u(s)\right\|_{L^2}^2ds},~\forall t>T.
\end{equation}
On the other hand, \eqref{GrindEQ__6_1_} and H\"{o}lder's inequality imply that
\begin{equation} \label{GrindEQ__6_44_}
|V_{\varphi_{R}}'(t)|=\left|2{\rm Im}\int{\varphi_{R}' \frac{x\cdot \nabla u}{r}\bar{u}dx}\right|\lesssim R \left\|\nabla u\right\|_{L^2}\left\| u\right\|_{L^2}\lesssim \left\|\nabla u\right\|_{L^2}.
\end{equation}
\eqref{GrindEQ__6_43_} and \eqref{GrindEQ__6_44_} show that
$$
\int_{T}^{t}{\left\|\nabla u(s)\right\|_{L^2}^2ds}\lesssim |V_{\varphi_{R}}'(t)| \lesssim \left\|\nabla u\right\|_{L^2}.
$$
Hence, defining $h(t):=\int_{T}^{t}{\left\|\nabla u(s)\right\|_{L^2}^2ds}$, we have $h^2(t)\lesssim h'(t)$. And this ODI has no global solution. Indeed, taking $T'>T$ and integrating on $[T', t)$, we get
$$
t-T'\lesssim \int_{T'}^{t}{\frac{h'(s)}{h^2(s)}ds}=\frac{1}{h(T')}-\frac{1}{h(t)}\le \frac{1}{h(T')}.
$$
This implies that $t<cT'+\frac{c}{h(T')}$.
This completes the proof.
\end{proof}

\end{document}